\documentclass[12pt, oneside]{amsart}
\usepackage[stable]{footmisc}
\makeatletter
\def\myfnt{\ifx\protect\@typeset@protect\expandafter\footnote\else\expandafter\@gobble\fi}
\makeatother
\usepackage{mathabx}
\usepackage{amsfonts}
 \usepackage{amssymb}
 \usepackage{amsmath,amsxtra,amsthm}
 \usepackage{pst-node}
\usepackage{fancyvrb}
\usepackage{cool}
\usepackage{array,epsfig}
\usepackage{amsfonts}
\usepackage{amssymb}
\usepackage{amsxtra}
\usepackage{amsthm}
\usepackage{mathrsfs}
\usepackage{color}
\usepackage{tikz}
\usepackage{mathtools}
\usepackage{tikz-cd}
\usepackage{booktabs}
\usepackage{siunitx}
\usepackage{soul}
\usetikzlibrary{positioning}
\usepackage{hyperref}
\usepackage{enumerate}
\usepackage{nomencl}
\makenomenclature
\usepackage[all,cmtip,matrix,arrow,curve]{xy}
\usepackage{mathrsfs}

\DeclareFontFamily{OT1}{pzc}{}
\DeclareFontShape{OT1}{pzc}{m}{it}{<-> s * [1.15] pzcmi7t}{}
\DeclareMathAlphabet{\mathpzc}{OT1}{pzc}{m}{it}
\DeclareSymbolFont{bbold}{U}{bbold}{m}{n}
\DeclareSymbolFontAlphabet{\mathbbold}{bbold}

\newtheorem{theorem}{Theorem}[section]
\newtheorem{theoremint}{Theorem}

\newtheorem{deff}[theorem]{Definition}

\newtheorem{example}[theorem]{Example}
\newtheorem{lemma}[theorem]{Lemma}

\newtheorem{cor}[theorem]{Corollary}
\newtheorem{prop}[theorem]{Proposition}
\newtheorem{rem}[theorem]{Remark}

\newcommand{\bqa}{\begin{align}}
\newcommand\eqa {\end{align}}
\newcommand{\beq}{\begin{align}}
\newcommand{\beqn}{\begin{align}\nonumber}
\newcommand{\eeq}{\end{align}}
\newcommand{\be}{\begin{array}}
\newcommand{\ee}{\end{array}}
\usepackage{amsthm,amssymb}

\newcommand{\la}{\langle}
\newcommand{\ra}{\rangle}

 \newcommand{\exd}{\mathrm{d}}

   \newcommand\vf\varphi

 \newcommand{\K}{\mathrm{K}}
 \newcommand{\s}{\mathrm{s}}
 \newcommand{\ttt}{\mathrm{t}}
 \newcommand{\m}{\mathrm{m}}
 \newcommand{\iii}{\mathrm{i}}
 \newcommand{\uuu}{\mathrm{u}}

\newcommand{\DD}{\mathbb{D}}

 \newcommand{\cC}{{\mathcal C}}

 \newcommand{\cO}{{\mathcal{O}}}

 \newcommand{\cL}{\mathcal{L}}
 
 \newcommand{\cG}{{\mathcal{G}}}
 
 \newcommand{\cF}{{\mathcal{F}}}

 \newcommand{\C}{{\mathbb C}}
 
 \newcommand{\R}{{\mathbb R}}

 \newcommand{\A}{{\mathbb A}}
 \newcommand{\N}{{\mathbb N}}
 \newcommand{\OO}{{\mathbb O}}
 \newcommand{\HH}{{\mathbb H}}

  \def\H{\mathbb H}

   \def\a{\alpha}

   \def\e{\epsilon}

\title[Minimal Lie groupoid and $\infty-$algebroid of the SOHF]{The minimal Lie groupoid and infinity algebroid \\ of the singular octonionic Hopf foliation}

\author{Hadi Nahari}
\email{hadi.nahari@u-pec.fr, stroblATmath.univ-lyon1.fr}

\author{Thomas Strobl}

\address{H.N.: Laboratoire d’Analyse et de Mathématiques Appliquées (LAMA), Universit\'e Paris Est Cr\'eteil,
61 Av. du Général de Gaulle, 94000 Créteil, France}
\vspace{50mm}
\address{T.S.: Institut Camille Jordan, Universit\'e Claude Bernard Lyon 1, Universit\'e de Lyon,
43 boulevard du 11 novembre 1918, 69622 Villeurbanne Cedex, France \qquad \& \newline Erwin Schr\"odinger International Institute for Mathematics and Physics, University of Vienna, Boltzmanngasse 9A, 1090 Wien, Austria}

\date{\today}
\usepackage{lipsum}
\begin{document}
\vspace{50mm}
\begin{abstract}

The famous singular leaf decomposition $\cL_{OH}$ of $\R^{16}\cong \OO^2$ induced by the Hopf construction for octonions $\OO$ has no known Lie group action generating it.
In this article we construct a  $\mathrm{G}_2$-equivariant Lie \emph{groupoid} $\cG \Rightarrow \OO^{2}$ whose orbits coincide with $\cL_{OH}$. Its Lie algebroid $E=\mathrm{Lie}(\cG)$ is of the form $\OO^4 \to \OO^2$ with polynomial structure functions. Its sheaf of sections induces a singular foliation $\cF_{OH} := \rho(\Gamma(E))$ on $\OO^{2}$, which we call the singular octonionic Hopf foliation (SOHF). $\cF_{OH}$ is shown to be maximal among all singular foliations $\cF$ generating $\cL_{OH}$---in the polynomial, the real analytic, as well as in the smooth setting.

\vskip 2mm \noindent  We extend $E$ to  a Lie $3$-algebroid, which is a minimal length representative of the universal Lie $\infty-$algebroid of the SOHF. 
This permits to prove that $E$ is the minimal rank Lie algebroid and that $\cG$  the lowest dimensional Lie groupoid which generate  the SOHF.

\vskip 2mm \noindent The leaf decomposition $\cL_{OH}$ is one of the few known examples of a singular Riemannian foliation in the sense of Molino which cannot be generated by local isometries (local non-homogeneity).
We improve this result by showing that any smooth singular foliation $\cF$ inducing $\cL_{OH}$ cannot be even Hausdorff Morita equivalent to a singular foliation $\cF_M$ on a Riemannian manifold $(M,g)$ generated by local isometries. Furthermore, we show that there is no real analytic singular foliation $\cF$ generating  $\cL_{OH}$ which  turns $(\R^{16}, g_{st}, \cF)$ into a module singular Riemannian foliation as defined in \cite{NS24}.

\end{abstract}

\maketitle
\tableofcontents
\newpage
\section{Introduction}
\addtocontents{toc}{\protect\setcounter{tocdepth}{1}}

\vskip 2mm\noindent Hopf fibrations are celebrated geometric objects which can be constructed for every normed division algebra. Hopf fibrations associated to the complex numbers and quaternions are well-known examples in differential geometry and mathematical physics \cite{H31, U03}. Due to the non-associative nature of octonions and the consequent challenges in calculations, the octonionic Hopf fibration on \(\mathrm{S}^{15} \subset \OO^2 \cong \R^{16}\) has been less studied in the literature. However, it possesses remarkable properties, highlighting the importance of non-associativity in its geometry \cite{OPPV13, BC21}. To construct the octonionic Hopf fibration, one has to start with the singular octonionic Hopf leaf decomposition $\cL_{OH}$.  Recall that the \emph{octonionic lines} in $\OO^2\cong\R^{16}$ \cite{GWZ86,OPPV13} are defined as

\begin{equation*}
         l_m:=\left\{(x,m\!\cdot\!x)\in \OO^2\,\colon\, x\in \OO\right\}.
\end{equation*}
for every $m\in \OO$, called the slope of the line, together with the octonionic line given by $l_{\infty}:=\left\{(0,x)\in \OO^2\,\colon\, \, x\in \OO\right\}$.

\begin{deff}
The \emph{singular Hopf leaf decomposition} $\cL_{OH}$ of $\OO^2$ is defined as the family of the leaves
\begin{align*}
    L_{m,r}&:=l_m \cap \mathrm{S}(r)\:\quad \forall  m\in \OO,\nonumber \\
    L_{\infty,r}&:=l_\infty \cap \mathrm{S}(r)\,,
\end{align*}
together with the origin in $\OO^2$. Here, $\mathrm{S}(r)$ is the sphere of radius $r>0$ in $\OO^2$.
\end{deff}

\vskip 2mm\noindent In particular, the leaves $L_{m,1}$ and $L_{\infty,1}$ correspond to the fibers of the octonionic Hopf fibration. 
In contrast to the complex and quaternionic Hopf fibrations, the octonionic Hopf fibration is \emph{non-homogeneous}, meaning there is no isometric Lie group action on the standard $\mathrm{S}^{15}$ that generates these fibers \cite{GWZ86, L93}. As a result, the leaf decomposition $\mathcal{L}_{OH}$, which forms a singular Riemannian foliation with respect to the standard metric on $\mathbb{R}^{16}$ in the sense of Molino \cite{M98}, is locally non-homogeneous near the origin \cite{MR19}.

\vskip 2mm\noindent 
The notion of module singular Riemannian foliations was introduced and developed in \cite{NS24} by adapting Molino's classical definition of singular Riemannian foliations \cite{M98} to the modern framework of singular foliations, defined as locally finitely generated involutive submodules of compactly supported vector fields \cite{AS09}: for a singular foliation $\cF \subset \mathfrak{X}_c(M)$ on a Riemannian manifold $(M,g)$, we have
\begin{deff}
    The triple $(M,g,\cF)$ is a \emph{module singular Riemannian} foliation if, for every vector field $X\in \cF$, we have

    \begin{equation}\label{msrfint}
        \cL_X g \in \Omega^1(M) \odot g_\flat(\cF)\,.
    \end{equation}
\end{deff}
Here, $g_\flat \colon TM \to T^*M$ is the musical isomorphism of vector bundles, given by $v \mapsto g(v, \cdot)$ for every vector $v\in TM$.

\vskip 2mm\noindent 
\vskip 2mm\noindent It was shown that the leaf decomposition of every module singular Riemannian foliation defines a singular Riemannian foliation in the sense of Molino. Moreover, for regular foliations, both definitions are equivalent. However, the following question had remained open:

\vskip 2mm \noindent \textbf{Question:} Let $\cF$ be a singular foliation on a Riemannian manifold $(M, g)$ such that its leaf decomposition defines a singular Riemannian foliation in the sense of Molino \cite{M98}. Is it possible to construct a module singular Riemannian foliation $(M, g, \cF')$ that has the same leaf decomposition as $\cF$?

\vskip 2mm \noindent We claim here that $\cL_{OH}$ partially provides a counterexample to this question:

\begin{theoremint}\label{ceint}
    Let $\cF$ be any singular foliation on the real analytic Riemannian manifold $(\OO^2,\exd s^2)$ with $\cL_{OH}$ as its leaf decomposition.  Then, the triple $(\OO^2, \exd s^2,\mathcal{F})$ induces a singular Riemannian foliation in the sense of Molino, but it is not a module singular Riemannian foliation.
\end{theoremint}

\vskip 2mm \noindent In addition, in this paper, we refine the classical result on the non-homogeneity of $\cL_{OH}$ as follows:

\begin{theoremint}\label{nonhomint}
    Let $\cF_0$ be \emph{any} singular foliation on $\OO^2$ having $\cL_{OH}$ as its leaf decomposition. Then $(\OO^2,\cF_0)$ is not Hausdorff Morita equivalent to any singular foliation $\cF$ on some Riemannian manifold $(M,g)$, whose leaf decompoistion is locally given by orbits of some isometric Lie group action. 
\end{theoremint}

\vskip 2mm\noindent A key step in proving Theorems \ref{ceint} and \ref{nonhomint} is the following central lemma, which characterizes the vector fields tangent to the leaves of $\cL_{OH}$: A vector field  $\renewcommand{\arraystretch}{0.7}\begin{pmatrix}u\\v\end{pmatrix}\in \mathfrak{X}(\OO^2)$  with $u,v\in C^\infty(\OO^2,\OO)$ is tangent to the leaves of $\cL_{OH}$, if and only if, for all  $(x,y)\in \OO^2\cong\mathbb{R}^{16}$
\begin{align}
         u\!\cdot \!\overline{y}+x\!\cdot\! \overline{v}&= \label{conditionint}0 \, \: , \\
         \langle x,u \rangle = \la y,v\ra&=0 \: . \label{orthogonalint}
\end{align}
 This lemma, in particular, implies that there are no linear vector fields tangent to $\cL_{OH}$, offering an alternative proof of the classical non-homogeneity result.

\vskip 2mm\noindent The local non-homogeneity of $\cL_{OH}$ implies the non-existence of an isometric Lie group action around the origin in $(\R^{16}, g_{st})$ that induces $\cL_{OH}$. Moreover, more generally, there is no known (unconstrained) Lie group action around the origin that generates these leaves. 
In this paper we construct a Lie groupoid $\cG\Rightarrow\OO^2$ which has $\cL_{OH}$ as its orbits---to the best of our knowledge, such a Lie groupoid has not been known as of now.

\vskip 2mm\noindent Its construction is based on considering the \emph{rescaling function} $\lambda\colon \OO^2\times\OO^2\to \R$ defined by the formula
    \begin{equation}
\label{lambdaint}
\lambda(F,G,x,y)=\sqrt{1+2\left(\la x,F\ra+\la y,G\ra+\la x\!\cdot\!\overline{y},F\!\cdot\!\overline{G}\ra\right)+\|x\|^2\,\|F\|^2+\|y\|^2\,\|G\|^2}\,,
\end{equation}
for all $(F,G)$ and $(x,y)\in \OO^2$. Then:
\begin{deff} \label{defgint} The arrow manifold of  $\cG\Rightarrow \OO^2$ is 
\begin{equation*}
\cG:=\OO^2\times\OO^2\setminus \left\{(F,G,x,y)\in\OO^2\times\OO^2\,\colon\,\lambda(F,G,x,y)= 0\right\}.
\end{equation*} 
Let $g \equiv(F,G,x,y)\in \cG$. Then $\s(g)=(x,y)$ and 
    \begin{equation*}
    \ttt(g)=\dfrac{1}{\lambda(g)}\left(x+\|x\|^2\,F+(x\!\cdot\!\overline{y})\!\cdot\!G\: ,\:y+\|y\|^2\,G+(y\!\cdot\!\overline{x})\!\cdot\!F\right) \, .
\end{equation*}
 The product of composable arrows is defined by:
\begin{equation*}  (F',G',x,',y')\!\cdot\!(F,G,x,y):=(F+\lambda(g)\!\cdot\!F',G+\lambda(g)\!\cdot\!G',x,y)  \, .
\end{equation*}
The unit map is $\uuu\colon \OO^2\to \cG\, , \; (x,y) \mapsto (0,0,x,y)$ and for the inverse one has 
\begin{equation*}
     (F,G,x,y)^{-1}=(-F/\lambda(g),-G/\lambda(g),\ttt(g)) \,.
\end{equation*}

\end{deff}

\vskip 2mm\noindent We then prove 

\begin{theoremint}
    $\cG$ is a $\mathrm{G}_2$-equivariant Lie groupoid, whose orbits are the leaves in $\cL_{OH}$.
\end{theoremint}

\vskip 2mm\noindent Differentiating this Lie groupoid to a Lie algebroid, we obtain the trivial vector bundle $E_0=\underline{\OO^2}$ over $\OO^2$ and the anchor map $\rho\colon E_0\to  T\OO^2\cong\underline{\OO^2}$ given by
\begin{align}
    \label{rhoanchorint}
\rho\renewcommand{\arraystretch}{0.7}\begin{pmatrix}u\\v\end{pmatrix}=\renewcommand{\arraystretch}{0.7}\begin{pmatrix}\|x\|^2u+(x\!\cdot\!\overline{y})\!\cdot\!v-(\la x,u\ra+\la y,v\ra)x\\\|y\|^2v+(y\!\cdot\!\overline{x})\!\cdot\!u-(\la x,u\ra+\la y,v\ra)y\end{pmatrix}\,.
        \end{align}
Here, $\underline{V}$ stands for a trivial vector bundle over the base manifold, which has the vector space $V$ as its fibers. The Lie bracket evaluated on constant sections $\renewcommand{\arraystretch}{0.7}\begin{pmatrix}u\\v\end{pmatrix},\renewcommand{\arraystretch}{0.7}\begin{pmatrix}u'\\v'\end{pmatrix}\in \Gamma(E_0)$ is defined by 
        \begin{align}
            \label{brabraint}[\begin{pmatrix}u\\v\end{pmatrix},\begin{pmatrix}u'\\v'\end{pmatrix}]&=(\la x,u\ra+\la y,v\ra)\begin{pmatrix}u'\\v'\end{pmatrix}-(\la x,u'\ra+\la y,v'\ra)\begin{pmatrix}u\\v\end{pmatrix}\,.        
        \end{align}

\vskip 2mm\noindent This Lie algebroid realizes $\cL_{OH}$ as leaf decomposition induced by a singular foliation $\cF_{OH}$. This singular foliation turns out to be maximal among all singular foliations with the same leaf decomposition. More precisely, utilizing Equations \eqref{conditionint} and \eqref{orthogonalint}, together with computations done by means of  Macaulay2, we prove the following result: The singular octonionic Hopf foliation $\cF_{OH}$ is generated by all vector fields tangent to the leaves of $\cL_{OH}$.

\vskip 2mm\noindent Finally, we complete the study of $\cF_{OH}$ by extending the Lie algebroid $(E_0,[\cdot,\cdot],\rho)$ to a universal Lie $3$-algebroid. The \emph{universal Lie $\infty$-algebroid} of a singular foliations is introduced in \cite{LGLS20}. Lie $\infty$-algebroids, first appeared in \cite{S05} (see also \cite{V10}) as higher analoguous of Lie algebroids, can be defined as a positively graded vector bundle $E = \bigoplus_{i \geq 0} E_{-i}$ over a manifold $M$, together with a family of graded skew-symmetric and multilinear maps $l_k \colon \wedge^k \Gamma(E) \to \Gamma(E)$ of degree $2 - k$, called $k$-brackets and an anchor $\rho\colon E_0\to TM$. The barckets and the anchor are required to satisfy some compatibility conditions. In \cite{LGLS20}, it is proven that for a foliated manifold $(M,\cF)$, upon the existence of a \emph{geomtric resolution}, one can associate a Lie $\infty$-algebroid to the singular foliation, called a universal Lie $\infty$-algebroid of $\cF$. This association turns out to be unique up to homotopy, and leads to invariants of singular foliations. However, this association is not constructive, and different singular foliations may need different techniques to compute the universal Lie $\infty$-algebroid. For example, for linear foliations obtained by the actions of some subgroups of the general linear group, the construction of the associated universal Lie $\infty$-algebroids is explained in \cite{S23}. 

\vskip 2mm\noindent To construct a universal Lie $3$-algebroid of $\cF_{OH}$, in the first step, we use an exact sequence found via Macaulay2. This leads to finding a geometric resolution for the singular foliation on $\OO^2$. There, we have $E_{0}:=\underline{\OO^2}, E_{-1}:=\underline{\R\oplus\OO\oplus\R}\,,$  $E_{-2}:=\underline{\R}\,$, and $E_i=0$ for $i\leq 3$. The anchor map $\rho\colon E_0\to T\OO^2\cong \underline{\OO^2}$ and the $2$-bracket restricted to sections of degree zero, coincide with the anchor and the Lie bracket of the Lie algebroid $(E_0,[\cdot,\cdot],\rho)$. The $1$-bracket on sections $\renewcommand{\arraystretch}{0.7}\begin{pmatrix}
    \mu \\
    a   \\
    \nu
\end{pmatrix}\in \Gamma(E_{-1})$ and $t\in \Gamma(E_{-2})$  is given by

\begin{align}\nonumber
    \exd^{(1)}\begin{pmatrix}
    \mu \\
    a   \\
    \nu 
\end{pmatrix}&:=\begin{pmatrix}
    \mu x+a\!\cdot\!y \\
    \nu y+\overline{a}\!\cdot\!x
\end{pmatrix}\,,\\
    \exd^{(2)}(t)&:=\begin{pmatrix}
    -\|y\|^2t \\
    (x\!\cdot\!\overline{y})t   \\
    -\|x\|^2t
\end{pmatrix}\,.\nonumber
\end{align}
The $2$-bracket on further sections is defined as 
\begin{align}
[\begin{pmatrix}u\\v\end{pmatrix},\begin{pmatrix}\mu\\a\\\nu\end{pmatrix}]&:=\begin{pmatrix}-2\la y,\overline{a}\!\cdot\!u\ra+2\la y,v\ra\mu\\x\!\cdot\!(\overline{u}\!\cdot\!a)+(a\!\cdot\!v)\!\cdot\!\overline{y}-\mu (x\!\cdot\!\overline{v})-\nu (u\!\cdot\!\overline{y})\\-2\la x,a\!\cdot\!v\ra+2\la x,u\ra\nu\end{pmatrix}\,,\nonumber\\\,
[\begin{pmatrix}u\\v\end{pmatrix},t]&:=2(\la x,u\ra+\la y,v\ra)t\,,\nonumber\\\,
[\begin{pmatrix}\mu\\a\\\nu\end{pmatrix},\begin{pmatrix}\mu'\\a'\\\nu'\end{pmatrix}]&:=4\la a,a'\ra-2\mu\nu'-2\mu'\nu\,.
\end{align}
All other brackets are set to be zero. 

\vskip 2mm\noindent In Proposition \ref{thmlie3} we prove that these data define a Lie $3$-algebroid, which is in addition minimal at the origin, i.e.\ all the $1$-brackets vanish at the origin. Using the results of \cite{LGLS20}, we prove:

\begin{theoremint}
    The Lie algebroid $(E_0, [\cdot, \cdot], \rho)$ and similarly the Lie groupoid $\mathcal{G} \Rightarrow \OO^2$ have the minimal dimension among Lie algebroids and Lie groupoids over $\OO^2$ which induce the singular foliation $\mathcal{F}_{OH}$.
\end{theoremint}

\vskip 2mm\noindent The structure of this paper is as follows$\colon$ 

\vskip 2mm\noindent In \emph{Section \ref{nda}}, we provide a self-contained background on the \emph{normed division algebras} and their properties. In particular, for the case of octonions, we mention many identities which turn out to be useful in calculations, in the absence of  associativity. We conclude the section by introducing the singular Hopf leaf decomposition associated to each of the four normed division algebras $\R,\C, \HH$, and $\OO$.

\vskip 2mm\noindent In \emph{Section \ref{secgroupoid}}, we provide and explain the construction of the Lie groupoid which induces the singular octonionic Hopf leaf decomposition as its orbits. We then differentiate it to a Lie algebroid of rank $16$, realizing $\cL_{OH}$ as the leaf decomposition of a singular foliation.

\vskip 2mm\noindent \emph{Section \ref{secnonhom}} is devoted to the study of $\cF_{OH}$, its non-homogeneity and maximality.

\vskip 2mm\noindent In \emph{Section \ref{sec:SRF}}, we study $\cF_{OH}$ as a singular Riemannian foliation in the sense of Molino \cite{M98} and in our sense. We explain how this example distinguishes the two definitions in the real analytic setting.

\vskip 2mm\noindent In
\emph{Section \ref{lieinf}}, finally, after recalling the Universal Lie $\infty$-algebroid of singular foliations, we construct the Universal Lie $3$-algebroid of $(\OO^2,\cF_{OH})$, and prove the minimality of the rank for the Lie algebroid $(E_0,[\cdot,\cdot],\rho)$, and correspondingly for the Lie groupoid $\cG\Rightarrow\OO^2$.

\vskip 2mm\noindent The computations done via Eisenbaud's Macaulay2 are explained in \emph{Appendix} \ref{appa}.

\subsection*{Acknowledgements}  

\noindent 
\vskip 2mm
\noindent

\noindent  Ricardo Mendes made us aware of the  singular octonionic leaf decomposition as  a non-homogenous singular Riemannian foliation (in the sense of Molino) and we thank him for this fruitful suggestion.

\vskip 2mm\noindent We are grateful to Martin Cederwall for drawing our attention to \cite{C92} and for disucssions about it. 

\vskip 2mm\noindent We also gratefully acknowledge stimulating discussions with Christian Blohmann and Camille Laurent-Gengoux. 

\vskip 2mm\noindent This work was supported by the LABEX MILYON (ANR-10-LABX-0070) of Universit\'e de Lyon, within the program ``Investissements d'Avenir'' (ANR-11-IDEX-0007) operated by the French National Research Agency (ANR). 

\vskip 2mm\noindent We also  profited from the marvellous environment provided  within the program "Higher structures and Field Theory" and "Geometry of gauge theories: old and new" at the ERWIN SCHR\"ODINGER INSTITUTE in Vienna. T.S.\ is furthermore very grateful to the ESI hosting him as a Senior Research Fellow during part of the time where this paper was finished.

\section{Normed division algebras and Hopf fibrations}\label{nda}

\subsection{Basic properties of normed division algebras}\label{basform}

\begin{deff}
    A \emph{normed division algebra} is a (finite-dimensional) Euclidean vector space $(\DD,\langle \cdot , \cdot \rangle)$ equipped with the structure of a unital $\R$-algebra satisfying 
    \begin{equation}
         \|a \!\cdot\! b\| = \|a\| \!\cdot\! \|b\| \label{norm}
    \end{equation}
for all $a,b \in \DD$, where the norm is the one induced by the inner product.
\end{deff}

\vskip 2mm \noindent Since \eqref{norm} implies $a \!\cdot\! b=0$ can be satisfied only if $a=0$ or $b=0$, this is a division algebra, as suggested by the name. The existence of a unit element $1$ implies the embedding $\iota \colon \R \to \DD, \alpha \mapsto \alpha 1$. The orthogonal projections of an element $a \in \DD$ to $\mathrm{im}(\iota)$ and $\mathrm{im}(\iota)^\perp$ permit us to define the \emph{real part}
and the  \emph{imaginary part} of $a$, respectively: $\mathrm{Re}(a) =\langle a, 1 \rangle 1 $ and $\mathrm{Im}(a)=a - \mathrm{Re}(a)$ and to generalize  complex conjugation by means of the involution:
\begin{equation}\label{conjdef}
   \overline{a} := \mathrm{Re}(a) - \mathrm{Im}(a) \, .
\end{equation}

\begin{theorem}[\cite{H98}] Every normed division algebra is isomorphic to one of the following four: The real numbers $\R$, the complex numbers $\C$, the quaternions $\HH$,  the octonions $\OO$.    
\end{theorem}

\vskip 2mm \noindent They can be obtained successively by the Cayley-Dickson construction starting from the real numbers, doubling the dimension in each step. In the process one looses in the first step that all elements are real, in the second step commutativity, and in the last step associativity. Continuing Cayley-Dickson further then violates Equation \eqref{norm}. While non-associative, the octonions are still an alternative algebra, i.e.\ the associator $[a,b,c] = a \!\cdot\! ( b \!\cdot\! c) - (a \!\cdot\! b) \!\cdot\! c$ is skew-symmetric for all $a,b,c \in \OO$. This implies in particular
\begin{equation}
    a \!\cdot\! (b \!\cdot\! a) = (a \!\cdot\! b) \!\cdot\! a \, ,
\end{equation}
which, for this reason, we will simply write as $a \!\cdot\! b \!\cdot\! a$ henceforth.

In addition, the real components of the arguments of the associator do not contribute. Thus one has, for example, $[a,b,c]=-[\overline{b},\overline{a},c]$, which, when written out, gives 
 \begin{equation}\label{semiasseqn}
        a \!\cdot\! (b \!\cdot\! c)+\overline{b} \!\cdot\! (\overline{a} \!\cdot\! c)=(a \!\cdot\! b) \!\cdot\! c + (\overline{b} \!\cdot\! \overline{a}) \!\cdot\! c\,.
    \end{equation}

 \vskip 2mm \noindent This equation yields several ones that we will use in the main text below and which we will derive from it now.

\vskip 2mm \noindent As a direct consequence of the Cayley-Dickson construction, we have $ \overline{a \!\cdot\! b} = \overline{b} \!\cdot\! \overline{a}$ for all $a,b\in \OO$ as well as
\begin{equation}
  \overline{a} \!\cdot\! a = a \!\cdot\! \overline{a} =  \|a\|^2 \, , \label{a^2}
\end{equation}
where the embedding $\iota$ is understood on the right-hand side. Replacing $a$ by $a+b$ in Equation \eqref{a^2} then recovers the standard inner product by the formula.
\begin{equation}\label{inpbar}
    \langle a , b\rangle=\tfrac{1}{2}(a \!\cdot\! \overline{b}+b \!\cdot\! \overline{a})=\mathrm{Re}(a \!\cdot\! \overline{b})\,.
\end{equation}
This in particular implies that Equation \eqref{semiasseqn} can be rewritten as
\begin{equation}\label{semiasseqn2}
        a \!\cdot\! (b \!\cdot\! c)+\overline{b} \!\cdot\! (\overline{a} \!\cdot\! c)=2\langle a , \overline{b}\rangle \!\cdot\! c\,.
    \end{equation}

\vskip 2mm \noindent Equation \eqref{a^2} implies that every $a\neq 0$ has an inverse $a^{-1}=\overline{a}/\|a\|^2$. If we replace $b$ and $c$ in Equation \eqref{semiasseqn2}  by $\overline{a}$ and $b$, respectively, we obtain  $a\!\cdot\! (\overline{a} \!\cdot\! b)=\|a\|^2 \!\cdot\! b$, which, for $a\neq 0$, implies 
\begin{align}
   a \!\cdot\! (a^{-1} \!\cdot\! b) = b\,,\label{invl}\\
   (b \!\cdot\! a^{-1}) \!\cdot\! a =b\,\label{invr},
\end{align}
where the second equation follows from the first one by conjugation.

\vskip 2mm \noindent Polarization of Equation \eqref{norm} leads to 
\begin{equation}
    \langle a\!\cdot\! b,a\!\cdot\! c\rangle=\|a\|^2 \langle  b, c\rangle=\langle b \!\cdot\! a, c \!\cdot\! a\rangle\,,
\end{equation}
which  yields the frequently used equations (for $a\neq 0$, replace $c$ by $a^{-1} \cdot c$ and by $c \cdot a^{-1}$, respectively):
\begin{align}
     \langle a\!\cdot\! b , c\rangle=\langle b , \overline{a} \!\cdot\! c\rangle\,,\label{switch1}\\
     \langle b \!\cdot\! a, c\rangle=\langle b , c \!\cdot\! \overline{a}\rangle\,.\label{switch2}
\end{align}

\vskip 2mm \noindent Replacing $c$ by $a$ in Equation \eqref{semiasseqn2}, we obtain the conjugation formula
\begin{equation}\label{conj}
    a \!\cdot\! b \!\cdot\! a = 2\langle a, \overline{b}\rangle a - \|a\|^2\overline{b}\,.
\end{equation}

\vskip 2mm \noindent Finally, these equations imply the well-known Moufang identities: for every $a,b,c\in \OO$ we have
    \begin{align}
        (a\!\cdot\! b)\!\cdot\!(c\!\cdot\! a) &= a\!\cdot\! (b\!\cdot\! c)\!\cdot\! a\,,\label{moufang1}\\
        a\!\cdot\!(b\!\cdot\! (a\!\cdot\! c)) &= (a\!\cdot\! b \!\cdot\! a)\!\cdot\! c\,,\label{moufang2}\\
        ((a\!\cdot\! b)\!\cdot\! c)\!\cdot\! b &= a\!\cdot\! (b\!\cdot\! c\!\cdot\! b)\,. \label{moufang3}
    \end{align}
Let us illustrate this for Equation \eqref{moufang1}: Using Equations \eqref{semiasseqn2} and \eqref{invr}, we have
\begin{align*}
       (a\!\cdot\! b)\!\cdot\!(c\!\cdot\! a)&=2\langle a\!\cdot\! b,\overline{c}\rangle\!\cdot\! a - \overline{c}\!\cdot\! ((\overline{b}\!\cdot\! \overline{a})\!\cdot\! a)\\
       &=2\langle a, (\overline{b\!\cdot\! c})\rangle \!\cdot\! a - \|a\|^2 \!\cdot\! (\overline{b\!\cdot\! c})\\
       &=a\!\cdot\! (b\!\cdot\! c) \!\cdot\! a\,,
   \end{align*}
where we used Equation \eqref{semiasseqn2} once more to obtain the last equality.

\subsection{Hopf fibrations}

\vskip 2mm \noindent The Hopf fibration can be constructed for each of the four normed division algebras $\R, \C, \HH$, and $\OO$. 
As explained in this subsection, there is a significant difference between the first three and the octonions, due to the lack of associativity for the latter. 

\vskip 2mm \noindent We start by describing the construction for the case of associative normed divisions algebras. Here $\A$ stands for $\R, \C$ or $\HH$. Associativity implies that the set of units $U=\left\{u\in \A \,\colon\, \|u\|^2=1\right\}$  in $\A$ forms a Lie group. As multiplication by elements of norm one preserves the norm of every element in $\A$, we can consider the right Lie group action of $U$ on the unit sphere $S\subset \A^2$, defined by
\begin{equation}
    (a,b) \!\cdot\! u:=(a \!\cdot\! u, b \!\cdot\! u)\,
\end{equation}
for every $(a,b)\in \A^2$ and $u\in U$. This action is free. It is proper as well since both $U$ and $S$ are compact manifolds. Consequently, the right Lie group action $S\times U \to S$ defines a principal $U$-bundle. Table \ref{T1} below provides an overview of the three corresponding bundles. 

\begin{table}[htb]
\centering

\begin{tabular}{|c|c|c|c|c|c|}
\hline
$\A$ & Structure group & Total space & Base manifold\\
\hline
$\R$ & $\mathbb{Z}/2\mathbb{Z}\cong S^0$ & $S^1$ & $S^1$\\
\hline
$\C$ & $U(1)\cong S^1$ & $S^3$ & $S^2$\\
\hline
$\HH$ & $SU(2)\cong S^3$ & $S^7$ & $S^4$\\
\hline
\end{tabular}
\caption{}
\label{T1}
\end{table}

\begin{rem}
 If we replace the right action in the construction of Hopf fibrations by a left action, we obtain the same fibrations for the cases of $\A=\R,\C$. However, due to non-commutativity, the case of $\A=\HH$ will have a different fibration. Although we again obtain a fibration of $S^7$ into $3$-spheres, these fibers are not identically the same as the fibers in the construction by the right action described above. To confirm this, note that in our description for any point $(a',b')$ on the fiber passing through $(a,b)\in S^7
    \subset \HH^2$, we have $b' \!\cdot\! \overline{a'}= b \!\cdot\! \overline{a}$, which is not true in the other description.
\end{rem}

\vskip 2mm \noindent Due to the lack of associativity, the set of unit octonions $S^7\subset \OO$ does not form a Lie group. Consequently, right multiplication by unit octonions is not a Lie group action and may fail to construct the octonionic Hopf fibration. To see the disadvantages of such a construction, it is useful to introduce the following more explicit description of octonions:

\begin{deff}\label{o2}
    The algebra $\OO$ of octonions is generated by an orthonormal basis $e_0 \equiv 1$ and $\left\{e_i\right\}_{i=1}^7$, where $e_0$ is chosen to be the unit and the multiplication among the remaining basis elements can be defined by 
\begin{equation}
  e_i\!\cdot \! e_j=-\delta_{ij}1+\epsilon_{ijk}e_k, \;\;\;\;\;\;\;\;\; i,j=1,...,7 \, .
\end{equation}
Here $\delta_{ij}$ is the Kronecker delta, and $\epsilon_{ijk}$ is a completely anti-symmetric tensor with value $1$ when $ijk\in\{123, 145, 176, 246, 257, 347, 365\}$ and $0$ for all other triples.
\end{deff}

\vskip 2mm \noindent In particular, Definition \ref{o2} identifies the underlying vector space of $\OO$ with $\R^8$, where the identification is given by $a=\sum_{i=0}^7a^ie_i\mapsto (a_0,a_1,\ldots,a_7)$. When this identification is understood, by abuse of notation, we write $a=(a_0,a_1,\ldots,a_7)=\left(a^i\right)_{i=0}^7$.

\vskip 2mm \noindent Now for every $(x,y)\in \mathrm{S}^{15}\subset \OO^2$ consider the $7$-sphere defined as 
$$S^7_{(x,y)}:=\left\{(x\!\cdot\!u, y\!\cdot\!u)\,\colon\, u\in\OO\,\,,\,\, \|u\|=1\right\}\subset \mathrm{S}^{15}\subset \OO^2\,.$$
If right multiplication by octonions of norm $1$ induces an honest fibration, for every point $(x',y')\in S^7_{(x,y)}$, we should have $S^7_{(x',y')}=S^7_{(x,y)}$. Equivalently, for every $(x,y)\in \mathrm{S}^{15}\subset \OO^2$ and every two unit octonions $u_1$ and $u_2$, there should exist a unique unit octonion $u_3$ such that
\begin{align}
    (x\!\cdot\!u_1)\!\cdot\!u_2=x\cdot u_3\,,\label{cont1}\\
    (y\!\cdot\!u_1)\!\cdot\!u_2=y\cdot u_3\,.\label{cont2}
\end{align}
But it is not difficult to find an example violating this necessary condition: If we choose $(x,y)=(e_1/\sqrt{2},e_2/\sqrt{2})$, $u_1=e_5$ and $u_2=e_4$, in Equation \eqref{cont1} we find the unique solution $u_3=-e_1$, but the unique solution for Equation \eqref{cont2} turns out to be $u_3=e_1$. Consequently, right multiplication by unit octonions does not fibrate $\mathrm{S}^{15}$ into $7$-spheres. In particular, it does not induce the octonionic Hopf fibration. 

\subsection{Singular Hopf leaf decomposition}

\vskip 2mm \noindent Since the usual construction fails for the octonionic Hopf fibration, we describe the construction of Hopf fibrations using $\DD$-lines, which can also be used in the case of octonions. 

\begin{deff}\label{lines}
    Let $\DD$ be a normed division algebra. For every $m\in \DD$, the \emph{$\DD$-line} in $\DD^2$  with slope $m$ is defined as 
    \begin{equation}
         l_m:=\left\{(x,m\!\cdot\!x)\in \DD^2\,\colon\, x\in \DD\right\}.
    \end{equation}
 One also defines $\DD$-line with slope $\infty$:
 \begin{equation*}
     l_{\infty}:=\left\{(0,x)\in \DD^2\,\colon\, \, x\in \DD\right\}\,.
 \end{equation*}
 
\end{deff}

\begin{deff}
Let $\DD$ be a normed division algebra. The \emph{singular Hopf leaf decomposition} of $\DD^2$ is defined as
\begin{align}
    L_{m,r}&:=l_m \cap \mathrm{S}(r)\:\quad \forall  m\in \DD,\nonumber \\
    L_{\infty,r}&:=l_\infty \cap \mathrm{S}(r)\,,
\end{align}
together with the origin in $\DD^2$. Here, $\mathrm{S}(r):=\left\{(x,y)\in \DD^2\,\colon\, \|x\|^2+\|y\|^2=r^2\right\}$ is the sphere of radius $r>0$ in $\DD^2$.
\end{deff}

\vskip 2mm \noindent The Hopf fibration associated to any of the normed division algebras $\DD=\R,\C,\HH$, and $\OO$ can be defined as the partition of the unit sphere $\mathrm{S}(1) \subset \DD^2$ into the leaves $L_{m,1}$, for all $m\in \DD$, and the leaf $L_{\infty,1}$.

\begin{rem}
    It can be easily verified that the definition of Hopf fibrations using $\A$-lines conincides with the fibrations described in Table \ref{T1}. Note that in the case of quaternions, if one prefers to work with left action, the definition of $\HH$-lines should be modified as
    \begin{equation}
        l_m:=\left\{(x,x\!\cdot\!m)\in \HH^2\,\colon\, x\in \H\right\}
    \end{equation}
    for the two fibrations to coincide.
\end{rem}

\section{The Lie groupoid and its induced Lie algebroid}\label{secgroupoid}
\subsection{The Lie groupoid}
\vskip 2mm \noindent In this section, we describe the constrution of a Lie groupoid $\cG \Rightarrow \OO^{2}$ which has $\cL_{OH}$ as its orbits. A prominent role in its construction will be played by the following function:
\begin{deff}\label{resfun}
    The \emph{(octonionic) rescaling function} $\lambda\colon \OO^2\times\OO^2\to \R$ is defined by the following formula:
    \begin{equation}
\label{lambda}
\lambda(F,G,x,y)=\sqrt{1+2\left(\la x,F\ra+\la y,G\ra+\la x\!\cdot\!\overline{y},F\!\cdot\!\overline{G}\ra\right)+\|x\|^2\, 
    \|F\|^2+\|y\|^2\, 
    \|G\|^2}\,.
\end{equation} 
\end{deff}

\vskip 2mm \noindent We refer to $F=\left(F^i\right)_{i=0}^7$,$G=\left(G^i\right)_{i=0}^7$ as \emph{arrow coordinates} and to $x=\left(x^i\right)_{i=0}^7$,$y=\left(y^i\right)_{i=0}^7$ as \emph{object coordinates}.

\vskip 2mm \noindent The following lemma ensures that the radicand in Definition \ref{resfun} is non-negative. Consequently, the rescaling function $\lambda$ is defined throughout $\OO^2\times \OO^2$.

\begin{lemma}\label{lemlresdef} For all $(F,G,x,y)\in \OO^2\times\OO^2$, one has
    \begin{align*}
        &\|x\|^2\left[1+2\left(\la x,F\ra+\la y,G\ra+\la x\!\cdot\!\overline{y},F\!\cdot\!\overline{G}\ra\right)+\|x\|^2\,\|F\|^2+\|y\|^2\,\|G\|^2\right] \\
        &=\|x+\|x\|^2\, F+(x\cdot\overline{y})\cdot G\|^2
    \end{align*}
\end{lemma}

\begin{proof}
We first expand the right-hand side of the identity:
\begin{align*} 
    &\|x+\|x\|^2\, F+(x\!\cdot\!\overline{y})\!\cdot\! G\|^2
    =\big \langle 
     x+\|x\|^2\, F+(x\!\cdot\!\overline{y})\!\cdot\! G,x+\|x\|^2\, F+(x\!\cdot\!\overline{y})\!\cdot\! G 
    \big \rangle=\\
    &=\|x\|^2+2\|x\|^2\,\big(\la x,F\ra+\la x,(x\!\cdot\!\overline{y})\!\cdot\!G\ra+\la (x\!\cdot\!\overline{y})\!\cdot\!G,F\ra\big)+\|x\|^4\,\|F\|^2+\|x\|^2\,\|y\|^2\,\|G\|^2\,.
\end{align*}
The result then follows by noting that $\la (x\!\cdot\!\overline{y})\!\cdot\!G,F\ra) = \la x\!\cdot\!\overline{y},F\!\cdot\!\overline{G}\ra$
due to Equation \eqref{switch2} and that 
$\la x,(x\!\cdot\!\overline{y})\!\cdot\!G\ra =\la y,G\ra $ as a consequence of Equations \eqref{switch1} and \eqref{invr}.
$\blacksquare$
\end{proof}

\begin{theorem} \label{defg} The following data define a $\mathrm{G}_2$-equivariant Lie groupoid $\cG\Rightarrow \OO^2$, whose orbits coincide with the singular octonionic Hopf leaf decomposition $\cL_{OH}$.
\begin{itemize}
    \item The manifold of arrows $\cG$ is an open subset of $\OO^2\times\OO^2$ given by 
    \begin{equation*}
    \cG:=\OO^2\times\OO^2\setminus \cC
    \end{equation*}
    where
    
    $\cC=\left\{(F,G,x,y)\in\OO^2\times\OO^2\,\colon\,\lambda(F,G,x,y)= 0\right\}$.
    \item For every arrow $g=(F,G,x,y)\in \cG$, the source map $\s\colon\cG\to\OO^2$ and the target map $\ttt\colon\cG\to\OO^2$ are given by:
    \begin{align*}
    \s(g)&=(x,y)\, ,\\
    \ttt(g)&=\dfrac{1}{\lambda(g)}\left(x+\|x\|^2\,F+(x\!\cdot\!\overline{y})\!\cdot\!G\: ,\:y+\|y\|^2\,G+(y\!\cdot\!\overline{x})\!\cdot\!F\right) \, .
\end{align*}
\item The multiplication map $\m\colon \cG^{(2)}:=\left\{(g',g)\in \cG\ \!\!\times \cG\colon\, \s(g')=\ttt(g)\right\}\to\cG\,,$ is defined as follows: The product $\m(g',g)\equiv g'\!\cdot\!g$  of composable arrows $g'=(F',G',x,',y')$ and $g=(F,G,x,y)$ 
is given by
\begin{equation*}  g'\!\cdot\!g:=(F+\lambda(g)\!\cdot\!F',G+\lambda(g)\!\cdot\!G',x,y)  \, .
\end{equation*}
\item The unit map $\uuu\colon \OO^2\to \cG$ is given by associating the arrow $1_{(x,y)}:=(0,0,x,y)$ to every object $(x,y)\in \OO^2$.
\item The inverse $\iii \colon \cG\to \cG$ applied to an arrow $g=(F,G,x,y)$ gives $g^{-1} \equiv  \iii(g)$ by means of 
\begin{equation*}
     g^{-1}=(-F/\lambda(g),-G/\lambda(g),\ttt(g)) \,.
\end{equation*}
\end{itemize}
\end{theorem}

\vskip 2mm \noindent We first establish some identities for the structure maps of $\cG$ in the following lemmas, which will be used in the proof of the theorem.

\begin{lemma}\label{sameslop}
For every arrow $g\in\cG$, both $\s(g)$ and $\ttt(g)$ belong to the same octonionic line.
\end{lemma}

\begin{proof}
    If $x=0$, both $\s(g)$ and $\ttt(g)$ belong to the octonionic line $l_\infty$. If $x$ is non-zero, put $m=y\!\cdot\!x^{-1}$. Then evidently  $\s(g)=(x,y) \in l_m$.   To show that also $\ttt(g)$ belongs to $l_m$, we first show that  $(\|x\|^2F,(y\!\cdot\!\overline{x})\!\cdot\!F)$ and $((x\!\cdot\!y)\!\cdot\!G,\|y\|^2G)$ belong to the octonionic line $l_m$:
\begin{equation*} ((y\!\cdot\!\overline{x})\!\cdot\!F)\!\cdot\!(\|x\|^2F)^{-1}=((y\!\cdot\!\overline{x})\!\cdot\!F)\!\cdot\!\frac{\overline{(\|x\|^2F)}}{\|x\|^4\!\cdot\!\|F\|^2}=(y\!\cdot\!\frac{x}{\|x\|^2})\!\cdot\!\frac{\|F\|^2}{\|F\|^2}\\
    =y\!\cdot\!x^{-1}  
\end{equation*}     
and
\begin{equation*} 
    (\|y\|^2G)\!\cdot\!((x\!\cdot\!\overline{y})\!\cdot\!G)^{-1}=(\|y\|^2G)\!\cdot\!\frac{\overline{G}\!\cdot\!(y\!\cdot\!\overline{x})}{\|x\|^2\|y\|^2\|G\|^2}
    =\frac{\|G\|^2}{\|G\|^2}\!\cdot\!(y\!\cdot\!\frac{\overline{y}}{\|x\|^2})
    =y\!\cdot\!x^{-1}\,.
\end{equation*}  
 As $l_m$ is a vector space, we obtain that also $\ttt(g)\in l_m$.\hfill $\blacksquare$
\end{proof}

\begin{cor}\label{corxy}
    For every arrow $g=(F,G,x,y)\in \cG$ with $\ttt(g)=(x',y')$, we have
    \begin{equation*}
        y\!\cdot\!\overline{x}=y'\!\cdot\!\overline{x'}\,.
    \end{equation*}
\end{cor}

\begin{lemma}\label{lammult}
    For composable arrows $(g',g)\in \cG^{(2)}$ the map $\lambda$ is multiplicative, i.e.\ it satisfies $\lambda(g'\!\cdot g)=\lambda(g')\cdot\!\lambda(g)$.
\end{lemma}

\begin{proof}
Using Lemma \ref{lemlresdef} and Corollary \ref{corxy}, we have
\begin{align*}
    \lambda(g')\!\cdot\!\lambda(g)&=\frac{\|x'+\|x'\|^2F'+(x'\!\cdot\!\overline{y'})\!\cdot\!G'\|}{\|x'\|}\!\cdot\!\lambda(g)\\
    &=\frac{\|\lambda(g)\!\cdot\!x'+\|x\|^2\lambda(g)\!\cdot\!F'+(x\!\cdot\!\overline{y})\!\cdot\!\lambda(g)\!\cdot\!G'\|}{\|x\|}\\
    &=\frac{\|(x+\|x\|^2F+(x\!\cdot\!\overline{y})\!\cdot\!G)+\|x\|^2\lambda(g)\!\cdot\!F'+(x\!\cdot\!\overline{y})\!\cdot\!\lambda(g)\!\cdot\!G'\|}{\|x\|}\\
    &=\frac{\|x+\|x\|^2\cdot (F+\lambda(g)\!\cdot\!F')+(x\!\cdot\!\overline{y})\!\cdot\! (G+\lambda(g)\!\cdot\!G')\|}{\|x\|}\\
    &=\lambda(F+\lambda(g)\!\cdot\!F',G+\lambda(g)\!\cdot\!G,x,y)=\lambda(g'\!\cdot\!g)\,
\end{align*}\hfill $\blacksquare$
\end{proof}

\begin{lemma}\label{corass}
    The multiplication map $\m\colon \cG^{(2)}\to\cG$ is associative.
\end{lemma}

\begin{proof}
    Consider the arrows $g'',g',g\in \cG$, such that $\s(g'')=\ttt(g'')$ and $\s(g')=\ttt(g)$. Using Lemma \ref{lammult}, we have
    \begin{align*}
    g''\!\cdot\!(g'\!\cdot\!g)&=g''\!\cdot\!(F+\lambda(g)\!\cdot\!F',G+\lambda(g)\!\cdot\!G',x,y)\\
    &=(F+\lambda(g)\!\cdot\!F'+\lambda(g'.g)\!\cdot\!F'',G+\lambda(g)\!\cdot\!G'+\lambda(g'\!\cdot\!g)\!\cdot\!G'',x,y)\\   
    &=(F+\lambda(g)\!\cdot\!F'+\lambda(g)\!\cdot\!\lambda(g')\!\cdot\!F'',G+\lambda(g)\!\cdot\!G'+\lambda(g)\!\cdot\!\lambda(g')\!\cdot\!G'',x,y)\\
    &=(F'+\lambda(g')\!\cdot\!F'',G'+\lambda(g')\!\cdot\!G'',x',y')\!\cdot\!g\\
    &=(g''\!\cdot\!g')\!\cdot\!g\,.
\end{align*}
\hfill $\blacksquare$
\end{proof}

\begin{lemma}\label{targsame}
    For composable arrows $(g',g)\in \cG^{(2)}$, we have $\ttt(g'\!\cdot g)=\ttt(g')$ and $\s(g'\!\cdot\!g)=\s(g)$.
\end{lemma}

\begin{proof}
Let $g'=(F',g',x',y')$ and $g=(F,G,x,y)$ with $(x',y')=t(F,G,x,y)$. Lemma \ref{lammult} together with Corollary \ref{corxy} and $\|x'\|=\|x\|$ and $\|y'\|=\|y\|$ imply that
\begin{align*}
    \ttt(g'\!\cdot\!g)&=\ttt(F+\lambda(g)\!\cdot\!F',G+\lambda(g)\!\cdot\!G,x,y)\\
    &=\frac{1}{\lambda(g')\!\cdot\!\lambda(g)}\left(x+\|x\|^2\,(F+\lambda(g)\!\cdot\!F')+(x\!\cdot\!\overline{y})\!\cdot\!(G+\lambda(g)\!\cdot\!G')\: ,\:y+\|y\|^2\,G+(y\!\cdot\!\overline{x})\!\cdot\!F\right)\\
    &=\frac{1}{\lambda(g')}\left(x'+\|x'\|^2F'+(x'\!\cdot\!\overline{y'})\!\cdot\!G'\: ,\:y'+\|y'\|^2\,G'+(y\!\cdot\!\overline{x})\!\cdot\!F'\right)=\ttt(g')\,.
\end{align*}
The second identity is obvious from the definition.
\hfill $\blacksquare$
\end{proof}

\begin{lemma}\label{lemtoi}
    For every arrow $g\in \cG$, one has $g^{-1}\!\cdot\!g=1_{\s(g)}$. 
    In addition, $\ttt \circ \iii = \s$.
\end{lemma}

\begin{proof}
For every arrow $g=(F,G,x,y)$ $g^{-1}\!\cdot\!g=(0,0,x,y)=1_{(x,y)}$ is evident by the definition.  Lemma \ref{targsame} implies, moreover, that $\ttt \circ \iii(g)=\ttt(g^{-1}\!\cdot\!g)=\ttt(1_{(x,y)})= \s(g)$ 
    \hfill $\blacksquare$
\end{proof}
\begin{cor} \label{cortoi}
    For every arrow $g=(F,G,x,y)$, the composition   $g\!\cdot\!g^{-1}$ is defined and equal to $1_{\ttt(g)}$.
\end{cor}
\begin{proof}
    The composability of two arrows is a direct consequence of Lemma \ref{lemtoi}. For $g\!\cdot\!g^{-1}=(0,0,\ttt(g))=1_{t(g)}$ use that $\lambda(g^{-1})=\tfrac{1}{\lambda(g)}$  by Lemma \ref{lammult}.  \hfill $\blacksquare$
\end{proof}

\vskip 2mm \noindent {\bf Proof of Theorem $\ref{defg}$.}
    It is clear that $\cG$ is a smooth manifold and the maps $\s,\ttt,\iii$ and $\uuu$ are smooth. In addition, Lemmas \ref{corass}, \ref{targsame} and \ref{lemtoi} imply that the structure maps satisfy the compatibility conditions in the definition of a groupoid. 
    
    \vskip 2mm \noindent The source map is obviously a surjective submersion. The target map is a surjection as well, cf.\ Lemma \ref{lemtoi}. 
    In order to prove that the target map $\ttt$ is also a submersion, it is sufficient to show that for every arrow $g\in \cG$,  there is a local section passing through it. Let $g=(F_0,G_0,x_0,y_0)\in \cG$ and consider its inverse arrow $g^{-1}=(F,G,x_1,y_1)$. There exists an open neighborhood $U\subset \OO^2$ with $(F,G,x,y)\in \cG$ for all $(x,y)\in U$. We now define a smooth map $\sigma\colon U\to \cG$ by means of $\sigma(x,y):=\iii(F,G,x,y)$. We have $\sigma(x_1,y_1)=g$ and
\begin{equation}
\ttt\circ\sigma(x,y)=\ttt\circ\iii(F,G,x,y)=\s(F,G,x,y)=(x,y)\,,
\end{equation}
where we used Lemma \ref{lemtoi} to obtain the last equality. This implies that $\sigma$ is a local section passing through $g$, and consequently, $\ttt$ is a submersion. As a result, $\cG^{(2)}$ is a smooth manifold and the multiplication map $\m$ is smooth.
This completes the proof that $\cG\Rightarrow\OO^2$ is a Lie groupoid.

\vskip 2mm \noindent It remains to prove that the orbits of $\cG\Rightarrow\OO^2$ coincide with the leaves in $\cL_{OH}$. According to Lemma \ref{lemlresdef} one has $\|\ttt(g)\|=\|\s(g)\|$ for all $g \in \cG$. Together with Lemma \ref{sameslop} this implies that the orbits of $\cG$ are included in the leaves of $\cL_{OH}$. In addition, for every $(x,m\!\cdot\!x)\in L_{m,r}$, we consider the arrow 
    \begin{equation*}
        g_{m,r}=(\frac{x-1}{\|x\|},0,\|x\|,m\!\cdot\!\|x\|)\in \cG
    \end{equation*}
    which satisfies $\s(g_{m,r})=(\|x\|,m\!\cdot\!\|x\|)$ and
    \begin{align*}
    \ttt(g_{m,r})&=\tfrac{1}{\lambda(g_{m,r})}(\|x\|+\|x\|\!\cdot\!(x-1),m\!\cdot\!\|x\|+m\!\cdot\!\|x\|\!\cdot\!(x-1))\\
    &=\tfrac{\|x\|}{\lambda(g_{m,r})}\!\cdot\!(x,m\!\cdot\!x)=(x,m\!\cdot\!x)\,.
    \end{align*}
The last equality holds since 
\begin{align*}
    \lambda(g_{m,r})&=\sqrt{1+2\la \|x\|,\frac{x-1}{\|x\|}\ra+\|x\|^2\frac{\|x-1\|^2}{\|x\|^2}}\\
    &=\sqrt{1+2\la 1, x-1\ra+\|x-1\|^2}\\
    &=\sqrt{\|1+(x-1)\|^2}=\|x\|\,.
\end{align*}
Similarly for $(0,y) \in L_{\infty,r}$: The source of $g=(0,\frac{y-1}{\|y\|,0,\|y\|}$ is $(0,\|y\|)$ while its target is $(0,y)$.
Consequently, all points in a leaf of $\cL_{OH}$ can be joined to a single point of the same leaf by an arrow in  $\cG\Rightarrow\OO^2$.

\vskip 2mm \noindent To prove the $\mathrm{G}_2$-equivariance of the Lie groupoid, note that $\mathrm{G}_2$ is the automorphism group of the octonion algebra $\mathbb{O}$ \cite{C14}, meaning that every element $A \in \mathrm{G}_2$ acts on $\mathbb{O}$ such that:  
\[
A(x \cdot y) = A(x) \cdot A(y), \quad \forall x, y \in \mathbb{O}.
\]  
This directly implies that $A$ also preserves the multiplicative identity $1$, the conjugation operation, and the standard inner product on $\mathbb{O}$. The group $\mathrm{G}_2$ acts component-wise on $\mathbb{O}^2$ and $\mathbb{O}^2 \times \mathbb{O}^2$. By construction, the rescaling map $\lambda$ is $\mathrm{G}_2$-invariant. Consequently, all the structure maps of the Lie groupoid $\mathcal{G}$ are $\mathrm{G}_2$-equivariant.  
\hfill $\blacksquare$   

\vskip 2mm \noindent The same construction can be applied for any of the  normed division algebras $\DD$, i.e.\ also for  $\A \in \{\R,\C, \HH\}$. In the latter cases it then yields a Lie groupoid $\mathcal{G}_\A\Rightarrow \A^2$ for which the orbits give precisely the leaves of the singular Hopf foliation  associated to the respective associative division algebra $\A$. The dimension of $\mathcal{G}_\A$ equals to $4\dim(\A)$, i.e.\ $4$, $8$, and $16$, respectively. This is  significantly bigger than the ones found for
the corresponding action Lie groupoids
$\mathbb{A}^2 \rtimes U_\A \Rightarrow \A^2$, where 
$U_\A=\left\{u\in\A\,\colon\,\|u\|=1\right\}$ denotes the group of unitary elements found in the first colomn of Table \ref{T1}; these Lie groupoids have  dimensions $2$, $5$, and $11$, respectively. This is in contrast to $\cG=\cG_\OO$ whose dimension $32$ turns out to be minimal (see Theorem \ref{mindim} below). 
The relation of the Lie groupoids $\cG_\A$ with the corresponding action groupoids is clarified in the following proposition. 

\begin{prop}
    Let $\mathbb{A}^2 \rtimes U_\A \Rightarrow \A^2$ be the action groupoid for the diagonal right action of unitary elements $U_\A$ on $\A^2$.
    Then the smooth map $\varphi \colon \cG_\A\to \mathbb{A}^2 \rtimes U_\A$ given by
\begin{equation*}
\varphi(F,G,x,y) := \left((x,y),\frac{1+\overline{x}\cdot F+\overline{y}\cdot G}{\|1+\overline{x}\cdot F+\overline{y}\cdot G\|}\right)\,,
\end{equation*}
defines a Lie groupoid morphism covering the identity on $\A^2$.
=
\end{prop}

\begin{proof} First we observe that the rescaling function $\lambda \colon \A^2 \times \A^2 \to \R$, defined by Equation \eqref{lambda}, simplifies drastically in the case when the division algebra is associative: For every $g=(F,G,x,y)\in \cG_\A$ we have
    \begin{align*}
\lambda(g)=\dfrac{\|x+\|x\|^2F+x\!\cdot\!\overline{y}\!\cdot\!G\|}{\|x\|}=\dfrac{\|x\!\cdot\!(1+\overline{x}\!\cdot\!F+\overline{y}\!\cdot\!G)\|}{\|x\|}
        =\|1+\overline{x}\!\cdot\!F+\overline{y}\!\cdot\!G\|\,.
    \end{align*}
 Similarly, for the target map one now finds: 
    \begin{align*}
        \ttt(g)&=\frac{1}{\lambda(g)}\left(x+\|x\|^2F+x\!\cdot\!\overline{y}\!\cdot\!G, y+\|y^2\|G+y\!\cdot\!\overline{y}\!\cdot\!F\right)\\
        &=\frac{1}{\lambda(g)}\left(x\!\cdot\!(1+\overline{x}\!\cdot\!F+\overline{y}\!\cdot\!G),y\!\cdot\!(1+\overline{x}\!\cdot\!F+\overline{y}\!\cdot\!G)\right)
        =(x,y)\!\cdot\!\frac{1+\overline{x}\!\cdot\!F+\overline{y}\!\cdot\!G}{\|1+\overline{x}\!\cdot\!F+\overline{y}\!\cdot\!G\|}\,,
    \end{align*}
which is evidently equal to the target of $\varphi(g)$ in the action groupoid. 

\vskip 2mm \noindent To prove that $\varphi$ is a morphism of Lie groupoids, it remains to show that it also preserves the  multiplication: Let $g'=(F',G',x',y')$ and $g=(F,G,x,y)$ be arrows such that $(x',y')=t(g)$. On the one hand, we have
\begin{align*}
    \varphi(g')\!\cdot\!\varphi(g)&=\left((x',y'),\frac{1+\overline{x'}\!\cdot\!F'+\overline{y'}\!\cdot\!G'}{\lambda(g')}\right)\!\cdot\!\left((x,y),\frac{1+\overline{x}\!\cdot\!F+\overline{y}\!\cdot\!G}{\lambda(g)}\right)\\
    &=\left((x,y),\frac{(1+\overline{x}\!\cdot\!F+\overline{y}\!\cdot\!G)\!\cdot\!(1+\overline{x'}\!\cdot\!F'+\overline{y'}\!\cdot\!G')}{\lambda(g)\!\cdot\!\lambda(g')}\right)\,
\end{align*}
and, on the other hand,
\begin{align*}  \varphi(g'\!\cdot\!g)=\left((x,y),\frac{1+\overline{x}\!\cdot\!(F+\lambda(g)\!\cdot\!F')+\overline{y}\!\cdot\!(G+\lambda(g)\!\cdot\!G')}{\lambda(g'\!\cdot\!g)}\right)\,.
\end{align*}
Using Lemma \ref{lammult}, to establish $\varphi(g'\!\cdot\!g)=\varphi(g')\!\cdot\!\varphi(g)$ it suffices to show that
\begin{equation*}
    (1+\overline{x}\!\cdot\!F+\overline{y}\!\cdot\!G)\!\cdot\!(1+\overline{x'}\!\cdot\!F'+\overline{y'}\!\cdot\!G')=1+\overline{x}\!\cdot\!(F+\lambda(g)\!\cdot\!F')+\overline{y}\!\cdot\!(G+\lambda(g)\!\cdot\!G')\,
\end{equation*}
or, equivalently, that
\begin{equation*}
    (1+\overline{x}\!\cdot\!F+\overline{y}\!\cdot\!G)\!\cdot\!(\overline{x'}\!\cdot\!F'+\overline{y'}\!\cdot\!G')=\lambda(g)\!\cdot\!(\overline{x}\!\cdot\!F'+\overline{y}\!\cdot\!G')\,.
\end{equation*}
But as $(1+\overline{x}\cdot F+\overline{y} \cdot G)^{-1}=\overline{(1+\overline{x}\cdot F+\overline{y}\cdot G)}/\lambda(g)^2$, the latter equation becomes equivalent to 

\begin{equation*}
(\overline{x'}\!\cdot\!F'+\overline{y'}\!\cdot\!G')=\overline{\left(\frac{1+\overline{x}\!\cdot\!F+\overline{y}\!\cdot\!G}{\lambda(g)}\right)}\!\cdot\!(\overline{x}\!\cdot\!F'+\overline{y}\!\cdot\!G')\,,
\end{equation*}
which holds true since
\begin{equation*}
    (x',y')=\left(x\!\cdot\!\left(\frac{1+\overline{x}\!\cdot\!F+\overline{y}\!\cdot\!G}{\lambda(g)}\right),y\!\cdot\!\left(\frac{1+\overline{x}\!\cdot\!F+\overline{y}\!\cdot\!G}{\lambda(g)}\right)\right)\,.
\end{equation*}

    \hfill $\blacksquare$
\end{proof}

\subsection{The induced Lie algebroid}\label{liealgoh}
\vskip 2mm \noindent In this subsection, we differentiate the Lie groupoid $\cG\Rightarrow\OO^2$ introduced in the previous subsection so as to obtain a corresponding Lie algebroid. We follow the conventions of \cite{CF11}.

\begin{prop}\label{coralg}
    The Lie algebroid associated to the Lie groupoid $\cG\Rightarrow\OO^2$ is given by
    \begin{itemize}
        \item The trivial vector bundle $E_0=\underline{\OO^2}$ over $\OO^2$.
        \vspace{2mm}
        \item The anchor map $\rho\colon E_0\to  T\OO^2\cong\underline{\OO^2}$:
        \begin{align}
    \label{rhoanchor}
\rho\renewcommand{\arraystretch}{0.7}\begin{pmatrix}u\\v\end{pmatrix}=\renewcommand{\arraystretch}{0.7}\begin{pmatrix}\|x\|^2u+(x\!\cdot\!\overline{y})\!\cdot\!v-(\la x,u\ra+\la y,v\ra)x\\\|y\|^2v+(y\!\cdot\!\overline{x})\!\cdot\!u-(\la x,u\ra+\la y,v\ra)y\end{pmatrix}\,,
        \end{align}
        for every $\renewcommand{\arraystretch}{0.7}\begin{pmatrix}u\\v\end{pmatrix}\in E_0$ based at $(x,y)\in \OO^2$. 
        \item The Lie bracket evaluated on constant sections $\renewcommand{\arraystretch}{0.7}\begin{pmatrix}u\\v\end{pmatrix},\renewcommand{\arraystretch}{0.7}\begin{pmatrix}u'\\v'\end{pmatrix}\in \Gamma(E_0)$: 
        \begin{align}
            \label{brabra}[\begin{pmatrix}u\\v\end{pmatrix},\begin{pmatrix}u'\\v'\end{pmatrix}]&=(\la x,u\ra+\la y,v\ra)\begin{pmatrix}u'\\v'\end{pmatrix}-(\la x,u'\ra+\la y,v'\ra)\begin{pmatrix}u\\v\end{pmatrix}\,.        
        \end{align}
    \end{itemize}
\end{prop}

\begin{proof}
The vector bundle $E_0$ can be identified with $\ker \exd s \vert_{u(\OO^2)} \subset \cG\vert_{u(\OO^2)}$.
As $s\colon\cG\to\OO^2$ is projection to the second component of $\cG\subset \OO^2\times\OO^2$, $E_0$ is the trivial, rank $16$ vector bundle $\underline{\OO^2} \equiv \OO^2 \times \OO^2$. Assuming that the trivialization is given by the constant global frame $\tfrac{\partial}{\partial F^i}$ and $\tfrac{\partial}{\partial G^i}$ for $i=0,1,\ldots,7$, we may identify every vector field 
\begin{equation}
   \la u,\partial_F\ra+\la v,\partial_G\ra:=\sum_{i=0}^7\left(u^i\tfrac{\partial}{\partial F^i}+v^i\tfrac{\partial}{\partial G^i}\right)\in\Gamma(E_0) 
\end{equation} with $\renewcommand{\arraystretch}{0.7}\begin{pmatrix}u\\v\end{pmatrix}\in \Gamma(\underline{\OO^2})$, where $u$ and $v$ can be thought of as $\OO$-valued functions on $\OO^2$.

\vskip 2mm\noindent To evaluate the anchor map $\rho:E_0\to  T\OO^2\cong\underline{\OO^2}$, it suffices to compute $\rho\renewcommand{\arraystretch}{0.7}\begin{pmatrix}e_i\\0\end{pmatrix}$ and $\rho\renewcommand{\arraystretch}{0.7}\begin{pmatrix}0\\e_i\end{pmatrix}$:  Let us represent the vector $\tfrac{\partial}{\partial F^i}\Big\vert_{1_{(x,y)}}$ as the velocity 
vector of the smooth curve $g(\tau):\tau\to(\tau\!\cdot\!e_i,0,x,y)\in \mathrm{S}^{-1}(x,y)$, for which we have $g(0)=1_{(x,y)}$ and $\Dot{g}(0)=\tfrac{\partial}{\partial F^i}\Big\vert_{1_{(x,y)}}$. On this curve, the rescaling function is given by
    \begin{align*}
\lambda(g(\tau))=\sqrt{1+2\tau\la x, e_i\ra+\tau^2\|x\|^2}
\end{align*}
which implies that
\begin{align} \label{derlambda}
    \frac{\exd}{\exd\tau}\!
\Big(\lambda(g(\tau))\Big)\Big\vert_{\tau=0}=x^i\,.
\end{align}
For every $i=0,\ldots,7$ we can compute $\rho\begin{pmatrix}e_i\\0\end{pmatrix}$ as
\begin{align*}
\rho\begin{pmatrix}e_i\\0\end{pmatrix}&=\exd\ttt(\tfrac{\partial}{\partial F^i})=\frac{\exd}{\exd \tau}\Big\vert_{\tau=0}\ttt(g(\tau))\\&=\frac{\exd}{\exd \tau}\Big\vert_{\tau=0}\left[\frac{1}{\lambda(g(\tau))}\begin{pmatrix}x+\tau\|x\|^2e_i\\y+\tau(y\!\cdot\!\overline{x})\!\cdot\!e_i\end{pmatrix}\right]\\
    &=-x^i\begin{pmatrix}x\\y\end{pmatrix}+\begin{pmatrix}\|x\|^2e_i\\(y\!\cdot\!\overline{x})\!\cdot\!e_i\end{pmatrix}\\
    &=\begin{pmatrix}\|x\|^2e_i-x_i\,x\\(y\!\cdot\!\overline{x})\!\cdot\!e_i-x_i\,y\end{pmatrix}\,.
\end{align*}
$\rho\renewcommand{\arraystretch}{0.7}\begin{pmatrix}0\\e_i\end{pmatrix}$ can be computed in a similar way. Equation \eqref{rhoanchor} then follows from the $C^\infty(\OO^2)$-linearity of $\rho$.

\vspace{3mm} \noindent Given the anchor map, it is sufficient to specify the Lie algebroid bracket $[\cdot,\cdot]\colon \Gamma(E_0)\wedge\Gamma(E_0)\to\Gamma(E_0)$  on constant sections: We first compute the Lie bracket of the right-invariant vector fields induced by $\tfrac{\partial}{\partial F^i}\Big\vert_{u(\OO^2)}$ and $\tfrac{\partial}{\partial G^i}\Big\vert_{u(\OO^2)}$ on $\cG$. Let $g=(F,G,x,y)\in \cG$ be an arrow with $\ttt(g)=(x',y')$. Consider the curve $g(\tau)\colon \tau\mapsto (\tau e_i,0,x',y')$ which satisfies $g(0)=1_{(x',y')}$ and $\Dot{g}(0)=\tfrac{\partial}{\partial F^i}\Big\vert_{1_{(x',y')}}$.

 The right-invariant vector field $\left(\tfrac{\partial}{\partial F^i}\right)^{\!R}\in \mathfrak{X}(\cG)$ induced by $\tfrac{\partial}{\partial F^i}\Big\vert_{u(\OO^2)}$ is given by
\begin{align*}
    \tfrac{\partial}{\partial F^i}^R(g)=\frac{\exd}{\exd\tau}\Big\vert_{\tau=0}\left[g(\tau)\!\cdot\!g\right]&=\frac{\exd}{\exd\tau}\Big\vert_{\tau=0}(F+\tau\lambda(g)\!\cdot\!e_i,G,x,y)\\
    &=\lambda(g)\tfrac{\partial}{\partial F^i}\Big\vert_g
\end{align*}
and similarly
\begin{align*}
    \tfrac{\partial}{\partial G^i}^R(g)=\lambda(g)\tfrac{\partial}{\partial G^i}\Big\vert_g\,.
\end{align*}
To calculate commutators of the above vector fields, we need to compute for example
\[ \Big(\left(\tfrac{\partial}{\partial F^i}\right)^{\!R} \cdot \lambda \Big)(1_{(x,y)}),  \]
a calculation we performed already in Equation \eqref{derlambda}. This implies in particular
\begin{align*}
[\tfrac{\partial}{\partial F^i}^R,\tfrac{\partial}{\partial F^j}^R]\Big\vert_{1_{(x,y)}}&=x^i\tfrac{\partial}{\partial F^j}^R\Big\vert_{1_{(x,y)}}-x^j\tfrac{\partial}{\partial F^i}^R\Big\vert_{1_{(x,y)}}=x^i\tfrac{\partial}{\partial F^j}\Big\vert_{1_{(x,y)}}-x^j\tfrac{\partial}{\partial F^i}\Big\vert_{1_{(x,y)}}.\end{align*}
In a similar fashion, one obtains 
\begin{align*}
[\tfrac{\partial}{\partial F^i}^R,\tfrac{\partial}{\partial G^j}^R]\Big\vert_{1_{(x,y)}}&
=x^i\tfrac{\partial}{\partial G^j}\Big\vert_{1_{(x,y)}}-y^j\tfrac{\partial}{\partial F^i}\Big\vert_{1_{(x,y)}}\\
[\tfrac{\partial}{\partial G^i}^R,\tfrac{\partial}{\partial G^j}^R]\Big\vert_{1_{(x,y)}}&
=y^i\tfrac{\partial}{\partial G^j}\Big\vert_{1_{(x,y)}}-y^j\tfrac{\partial}{\partial G^i}\Big\vert_{1_{(x,y)}}\,.
\end{align*}
Equation \eqref{brabra} now follows by $\R$-linearity of the bracket.
    \hfill $\blacksquare$
\end{proof}

\section{The singular foliation and non-homogeneity}\label{secnonhom}

\vskip 2mm\noindent The Lie algebroid introduced in the previous section realizes $\cL_{OH}$ as the leaf decomposition of a \emph{singular foliation} on $\OO^2 \cong \R^{16}$. In the following subsection we provide some background on singular foliations (for more details see \cite{LGLR22}). We will do so using the original definition of singular foliations as introduced in \cite{AS09}. 
It provides the setting in which we will prove our result on the non-homogeneity of the singular octonionic Hopf foliation, see Theorem \ref{nonhom} below.
 
\vskip 2mm
\noindent
In Section \ref{sec:SRF}
we will need a generalization \cite{LGLS20} of this notion of a singular foliation to the real-analytic and polynomial setting, provided in Definition \ref{sheafSF} blow. 
The two definitions are  equivalent in the smooth setting \cite{GZ19}.

\subsection{Basic definitions and properties of singular foliations} \label{SFbasic}

Below, $\mathfrak{X}_c(M)$ stands for the $C^\infty(M)$-module of compactly supported vector fields on a smooth manifold $M$.

\begin{deff}\label{lfg}
A $C^\infty(M)$-submodule $\mathcal{F}\subset\mathfrak{X}_c(M)$ is said to be \emph{locally finitely generated} if, for every point $q\in M$, there exists an open neighborhood $U\subset M$ containing $q$ such that the submodule 
    $\iota_U^{-1}(\mathcal{F}):=\left\{X\vert_U\,\colon\,X\in \mathcal{F}\,,\,\mathrm{supp}(X)\subset U\right\}\,$ of $\mathfrak{X}_c(U)$ is finitely generated. This means that there exist finitely many (not necessarily compactly supported) 
vector fields $X_1,\ldots,X_N\subset \mathfrak{X}(U)$ 
such that 
\begin{equation}
\iota_U^{-1}(\mathcal{F})=\langle X_1,\ldots,X_N\rangle_{C^\infty_c(U)}\,.
\end{equation}
\end{deff}

\begin{deff} 
A \emph{singular foliation} (SF for short) on $M$ is a locally finitely generated $\mathcal{F}\subset\mathfrak{X}_c(M)$ closed with respect to the Lie bracket.  The pair $(M,\mathcal{F})$ is a  \emph{foliated manifold}. 
\label{SF}
\end{deff}

\vskip 2mm\noindent For every foliated manifold $(M,\mathcal{F})$ the leaf $L_q$ passing through $q\in M$ is given by the set of points which can be joined to $q$ by the flows of finitely many vector fields in $\cF$.

\begin{theorem}[\cite{H62}]\label{hermann} Let $(M,\cF)$ be a foliated manifold. Then for every point $q\in M$, the subset $L_q$ is an injectively immersed submanifold of $M$. 
\end{theorem}

\vskip 2mm\noindent 
The \emph{space of leaves} of $(M,\cF)$, i.e.\ the quotient space obtained form the equivalence relation of belonging to the same leaf, is denoted by $M/\cF$.

An SF contains more information than the leaf decomposition:
\begin{example} \label{k}
Consider $M=\mathbb{R}$ and fix some $k\in \mathbb{N}$. Then the vector fields vanishing  at least to order $k$ at the origin form a singular foliation $\mathcal{F}_k$. While the leaf decomposition is the same for all $k$, one has $\mathcal{F}_{k+1} \subsetneq\mathcal{F}_k$.

\end{example}
\vskip 2mm An important class of SFs arises from Lie algebroids: Let $E \to M$ be a Lie algebroid with anchor map $\rho$, then \begin{equation}
        \mathcal{F}_E:= \langle\left\{\rho(a)\,\colon\, a\in \Gamma(E)\right\}\rangle_{C^\infty_c(M)}\,
    \end{equation}
defines an SF on $M$. This applies also to the Lie algebroid constructed in the previous section.

\begin{deff}\label{deffoh}
    The $C^\infty(\OO^2)$-module $\cF_{OH}$ generated by elements of $\rho(\Gamma(E_0))$  for the Lie algebroid given in Proposition \ref{coralg} defines a singular foliation on $\OO^2$, referred to as the \emph{singular octonionic Hopf foliation}.
\end{deff}

\begin{deff}
Let $(M,\mathcal{F})$ be a foliated manifold. For every point $q\in M$, the \emph{fiber of $\mathcal{F}$} at $q$ is defined as:
\begin{equation*}
\mathcal{F}_q:=\mathcal{F}/\mathrm{I}_q\!\cdot\!\mathcal{F}
\end{equation*}
where $\mathrm{I}_q:=\left\{f\in C^{\infty}(M)\,\colon\,f(q)=0\right\}$ is the vanishing ideal of $q$ in $C^{\infty}(M)$.
\end{deff}

\begin{prop}\label{mingen}[\cite{AS09}] 
   The dimension of the vector space $\mathcal{F}_q$ is equal to the minimal number of locally generating vector fields for $\mathcal{F}$ in a small enough neighborhood of $q\in M$.
\end{prop}

\vskip 2mm\noindent For every point $q\in M$, the evaluation map $\mathcal{F}\to F_q$ at $q$, vanishes upon restriction to $\mathrm{I}_q\!\cdot\!\mathcal{F}$ and consequently induces the linear map $\mathrm{ev}_q\colon \mathcal{F}_q\to F_q$, $[X] \mapsto X\vert_q$. The map $\mathrm{ev}_q$ induces the following short exact sequence:

\begin{equation}\label{short}
    0\to \ker(\mathrm{ev}_q)\xhookrightarrow{} \mathcal{F}_q\xrightarrow{ev_q} F_q\to 0\, .
\end{equation}

\vskip 2mm\noindent The Lie bracket of vector fields restricted to $\mathcal{F}$ descends to a Lie bracket on the finite-dimensional vector space $\ker(ev_q)\subset \mathcal{F}_q$.

\begin{deff}
The vector space $\mathfrak{g}_q^{\mathcal{F}}:=\ker(ev_q)$ together with the  induced Lie bracket, defines the \emph{isotropy Lie algebra of $\mathcal{F}$} at $q$.
\end{deff}

\begin{example}
Let $n,k\in \N$,  $M=\mathbb{R}^n$, and let  $\mathcal{F}_k$ be the singular foliation generated by vector fields that vanish at least of  order $k$ at the origin. Then $\dim (\mathfrak{g}_0^{\mathcal{F}_k})=\binom{k+n-1}{n-1}$ and, for all $q \neq 0$,  $\dim (\mathfrak{g}_q^{\mathcal{F}_k})=0$.
\end{example}

\begin{deff}[\cite{AS09,GZ19}]\label{transversemap}
Let $f\colon N\to M$ be a smooth map and let $\cF$ be an SF on $M$. Then $f$ is said to be \emph{transverse to $\cF$}, if for every $q\in M$ one has
\begin{equation*}  \exd_qf(T_qN)+F_{f(q)}=T_{f(q)}M\,.
\end{equation*}
\end{deff}

\begin{example}
\label{inclusionSF}
 If $S\subset M$ is transverse to the leaves of a foliated manifold $(M,\cF)$, then the inclusion map $\iota_S\colon S\xhookrightarrow{} M$ is transverse to $\cF$.    
 \end{example} 
    \begin{example} A submersion $\pi\colon N\to M$ is transverse to every SF on $M$.
\end{example}

\begin{prop}[\cite{AS09}]\label{defpullback}
    Let $(M,\cF)$ be a foliated manifold and let $f\colon N\to M$ 
    be a smooth map transverse to $\cF$. Then the $C^\infty(N)$-module $f^{-1}\cF$ generated by vector fields on $N$ projectable to $\cF$ is a singular foliation on $N$.
\end{prop}

\vskip 2mm\noindent Here, a vector field $V\in \mathfrak{X}(N)$ is called \emph{projectable to $\cF$}, if there exists a vector field $X\in \cF$ such that for every point $q\in N$ we have:
\begin{equation*}
    \exd _qf(V\vert_q)=X\vert_{f(q)}\,.
\end{equation*}

\begin{deff}[\cite{GZ19}]\label{hmedef}
Two foliated manifolds $(M_1,\mathcal{F}_1)$ and $(M_2,\mathcal{F}_2)$ are \emph{Hausdorff Morita equivalent} if there exists a smooth manifold $N$ and surjective submersions with connected fibers $\pi_i \colon N\to M_i$, $i=1,2$, such that:
\begin{equation*}
    \pi_1^{-1}\mathcal{F}_1=\pi_2^{-1}\mathcal{F}_2\,.
\end{equation*}
\noindent In this case we write $(M_1,\mathcal{F}_1)\sim_{ME}(M_2,\mathcal{F}_2)$.
\end{deff}

\begin{example}
    Let $(M_1,\cF_1)$ and $(M_2,\cF_2)$ be isomorphic foliated manifolds, i.e.\ there exists a diffeomorphism $\Phi\colon M_1 \to M_2$ satisfying $\Phi_*(\cF_1)=\cF_2$. Then we have $(M_1,\mathcal{F}_1)\sim_{ME}(M_2,\mathcal{F}_2)$ by choosing  $N=M_1$, $\pi_1=\mathrm{Id}_{M_1}$ and $\pi_2=\Phi$  in Definition \ref{hmedef}.
\end{example}

\begin{theorem}[\cite{GZ19}]\label{hme} Let $(M_1,\mathcal{F}_1)$ and $(M_2,\mathcal{F}_2)$ be Hausdorff Morita equivalent foliated manifolds. Then:
\vskip 2mm\noindent (i) There is a canonical homeomorphism between the  leaf spaces $M_1/\cF_1$ and $M_2/\cF_2$. 
\vskip 2mm\noindent(ii) Let $q_1\in M_1$ and $q_2\in M_2$ be points in corresponding leaves. Choose transversal slices $S_{q_1}$ at $q_1$ and $S_{q_2}$ at $q_2$. Then the foliated manifolds $(S_{q_1},\iota_{S_{q_1}}^{-1}\mathcal{F}_1)$ and $(S_{q_2},\iota_{S_{q_2}}^{-1}\mathcal{F}_2)$ as well as the isotropy Lie algebras $\mathfrak{g}_{q_1}^{\mathcal{F}_1}$ and $\mathfrak{g}_{q_2}^{\mathcal{F}_2}$ are both isomorphic.
\end{theorem}

\subsection{Non-homogeneity}

\vskip 2mm\noindent 
We start by characterizing vector fields tangent to the leaves in $\cF_{OH}$ as solutions of a system of functional equations. Throughout this section, every vector field \[\la u,\tfrac{\partial}{\partial x}\ra+\la v,\tfrac{\partial}{\partial y}\ra=\sum_{i=1}^7(u^i\partial x^i+v^i\partial y^i)\] is identified with $\renewcommand{\arraystretch}{0.7}\begin{pmatrix}u\\v\end{pmatrix}\in \Gamma(\underline{\OO^2})$, where $u$ and $v$ can be thought of as $\OO$-valued functions on $\OO^2$.

\vskip 2mm \noindent \begin{lemma}\label{lemchar} A vector field  $\renewcommand{\arraystretch}{0.7}\begin{pmatrix}u\\v\end{pmatrix}\in \mathfrak{X}(\OO^2)$ is tangent to the leaves of $\cL_{OH}$, if and only if
\begin{align}
         u\!\cdot \!\overline{y}+x\!\cdot\! \overline{v}&= \label{condition}0 \, , \\
         \langle x,u \rangle = \la y,v\ra&=0 \,  .\label{orthogonal}
\end{align}
for all  $(x,y)\in \OO^2\cong\mathbb{R}^{16}$. 
\end{lemma}

\vskip 2mm \noindent \begin{proof} Assume that $\renewcommand{\arraystretch}{0.7}\begin{pmatrix}u\\v\end{pmatrix}$ is tangent to all the leaves in $\cL_{OH}$. For $x=0$, the vector field $\renewcommand{\arraystretch}{0.7}\begin{pmatrix}u\\v\end{pmatrix}$ being tangent to the octonionic line $l_\infty$ gives $u=0$, implying Equation \eqref{condition}. As the vector field is also tangent to the spheres $\mathrm{S}(r)$ for every $r\geq 0$ and $\la x,u\ra=0$, one obtains Equation \eqref{orthogonal}.

\vskip 2mm \noindent For $(x,y)=(x_0,y_0)$ with $x_0\neq 0$, consider the integral curve $\gamma\colon \R \rightarrow \OO^2$ of the vector field $\renewcommand{\arraystretch}{0.7}\begin{pmatrix}u\\v\end{pmatrix}$ satisfying $\gamma (0)=(x_0,y_0)$. Writing $\gamma(t) = (x(t),y(t))$, one has $(x(0),y(0))=(x_0,y_0)$, $\Dot{x}(t)=u(x(t),y(t))$, and $\Dot{y}(t)=v(x(t),y(t))$. The curve $\gamma$ is contained in the octonionic line $l_m$ with $m=y_0\! \cdot \!x_0^{-1}$. This implies that $m=y(t)\! \cdot \!x(t)^{-1}$ for all $t$. Consequently, $y(t)=m\! \cdot \!x(t)$, and differentiation with respect to $t$ gives $v(x(t),y(t))=m\! \cdot \! u(x(t),y(t))$. Evaluation at $t=0$ yields

\begin{equation}
      (y\! \cdot \! x^{-1})\! \cdot \! u=v \label{cond}\,.
\end{equation}

\vskip 2mm \noindent Taking the inner product of both sides with $y$ and using the identities \eqref{switch1} and \eqref{invr}, we have 
\begin{equation}
    \langle y,v \rangle = \langle y, (y\! \cdot \! x^{-1})\! \cdot \! u \rangle = \tfrac{1}{\|x\|^2}\langle (x\! \cdot \! \overline{y})\! \cdot \! y, u \rangle =\tfrac{\|y\|^2}{\|x\|^2} \langle x,u \rangle
\end{equation}
 
\vskip 2mm \noindent Since, in addition, $\langle x,u \ra + \la y,v \ra= (1+\tfrac{\|y\|^2}{\|x\|^2})\la x,u\ra=0$ by $\renewcommand{\arraystretch}{0.7}\begin{pmatrix}u\\v\end{pmatrix}$ being tangent to the spheres $\mathrm{S}^{15}(r)$, Equation \eqref{orthogonal} follows.

\vskip 2mm \noindent Multiplying both sides of Eq.\ \eqref{cond} by $x^{-1}$ from the right and using the third Moufang identity\eqref{moufang3}, we obtain

\begin{equation*}
    y  \!  \cdot  \! (x^{-1}  \! \cdot \! u  \! \cdot  \! x^{-1} ) = v \!  \cdot  \! x^{-1} \, .
\end{equation*}
Multiplying both sides by $\|x\|^4$ and using Equation \eqref{conj} imply that

\begin{equation*}
    \|x\|^2v \! \cdot  \!\Bar{x}=y  \!  \cdot  \! (\overline{x}  \! \cdot \! u  \! \cdot  \! \overline{x} )=y \!\cdot \!(2 \la \Bar{x},\Bar{u} \ra \Bar{x}-\|x\|^2\Bar{u})=-\|x\|^2y\!\cdot\!\overline{u} \, ,
\end{equation*}
which gives $ v\!\cdot\!\Bar{x}=-y \!\cdot\!\Bar{u}$ or, by conjugation, to Eq.\ \eqref{condition}.

 \vskip 2mm \noindent Conversely, assume that $\renewcommand{\arraystretch}{0.7}\begin{pmatrix}u\\v\end{pmatrix}$ satisfies Equations \eqref{condition} and \eqref{orthogonal}. These equation together with Equation \eqref{conj} and the third Moufang identity \eqref{moufang3} imply Equation \eqref{cond}, 
 by following the above calculation from back to front. 
 Now, consider some integral curve $\gamma(t) = (x(t),y(t))$. Since $\Dot{x}(t)=u(x(t),y(t))$ and $\Dot{y}(t)=v(x(t),y(t))$, Equation \eqref{orthogonal} implies that $\|x(t)\|^2$ and $\|y(t)\|^2$ are constant in $t$. As a first consequence, if $\gamma$ passes through the origin it must be a constant curve, and if it intersects $L_{\infty,r}$, it stays in $L_{\infty,r}$. If $x(t_0)\neq 0$ for some $t_0$, then one has $x(t)\neq 0$ for all $t$. Differentiation with respect to $t$ of the equation $(y(t)\!\cdot \!x(t)^{-1})\!\cdot\! x(t)=y(t)$ gives

 \begin{equation}\label{convcond}
     0=(y(t)\!\cdot\!x(t)^{-1})\!\cdot\!\Dot{x}(t)-\Dot{y}(t)+\left[\frac{\exd}{\exd t}(y(t)\!\cdot\!x(t)^{-1})\right]\!\cdot\!x(t)=\left[\frac{\exd}{\exd t}(y(t)\!\cdot\!x(t)^{-1})\right]\!\cdot\!x(t)\,,
 \end{equation}
 where we used Equation \eqref{cond} in the last equality. Since $x(t)\neq 0$, this implies that $y(t)\!\cdot\!x(t)^{-1}$ equals to a constant $m\in \OO$ and in particular $\gamma(t)=(x(t),m\!\cdot\!x(t))$ lies in the octonionic line $l_m$ for all $t$. In addition, Equation \eqref{orthogonal} ensures that it lies in some leaf of $\cL_{OH}$. We have shown that, if the integral curve $\gamma$ intersects a leaf in $\cL_{OH}$, it stays in that leaf. As a result, the vector field $\renewcommand{\arraystretch}{0.7}\begin{pmatrix}u\\v\end{pmatrix}$ is tangent to the leaves of $\cL_{OH}$.\hfill $\blacksquare$
\end{proof}

\begin{prop}\label{FOHmax}
    The singular octonionic Hopf foliation $\cF_{OH}$ is generated by all vector fields tangent to the leaves of $\cL_{OH}$.
\end{prop}

\begin{proof}
    Using the exact sequence described in Appendix \ref{appa} for the case of $C^\infty(M)$-modules, kernel of the map $\mathrm{J}\colon \Gamma(\underline{\OO^2})\to \Gamma(\underline{\R\oplus\OO\oplus\R})$ given by
    \begin{equation*}
        \mathrm{J}\begin{pmatrix}u\\v\end{pmatrix}=\begin{pmatrix}\la x,u\ra\\u\!\cdot\!\overline{y}+x\!\cdot\!\overline{v}\\\la y,v\ra\end{pmatrix}
    \end{equation*}
is generated by the image of the anchor $\rho\colon \Gamma(E_0)\to \mathfrak{X}(\OO^2)\cong \Gamma(\underline{\OO^2})$. Lemma \ref{lemchar} implies the result.\hfill $\blacksquare$
\end{proof}

\vskip 2mm\noindent This characterization of the vector fields tangent to the leaves in $\cL_{OH}$, together with the non-associativity of octonions, gives rise to the following important lemma:

\begin{lemma}\label{nolin}
Let $\renewcommand{\arraystretch}{0.7}\renewcommand{\arraystretch}{0.7}\begin{pmatrix}u\\v\end{pmatrix}\in \mathfrak{X}(\OO^2)$ be a vector field with  $u,v\in C^\infty(\OO^2,\OO)$ linear in coordinates $x^i$ and $y^i$. If $\renewcommand{\arraystretch}{0.7}\renewcommand{\arraystretch}{0.7}\begin{pmatrix}u\\v\end{pmatrix}$ is tangent to $\cL_{OH}$, then $u=v\equiv 0$.
\end{lemma}

\vskip 2mm\noindent \begin{proof} We prove the statement by contradiction. Assume that $X$ is a non-zero linear vector field on $\OO^2\cong\R^{16}$, tangent to the the leaves of $\cL_{OH}$. Equation \eqref{condition} then implies that $u$ is forced to be independent of $y$ and $v$ be independent of $x$. Consequently, there are $\R$-linear maps $A,B \colon \OO \to \OO$ such that $u(x,y)=Ax$ and $v(x,y)=By$. Now, Equation \eqref{condition} and \eqref{orthogonal} can be rewritten as
\begin{eqnarray}
         (Ax)\! \cdot \! \Bar{y}+x\! \cdot \! \overline{(By)}&= \label{condition2}
         &0 \, , \\
         \langle Ax,x \rangle &=& 0 \,  .\label{orthogonal2}
\end{eqnarray}

\vskip 2mm\noindent Choosing $y=1$ in Equation \eqref{condition2}  gives
\begin{equation*}
 Ax = -x\! \cdot \! \overline{(B1)}\,,
 \end{equation*}
which together with  Equation \eqref{orthogonal2} imply that $B1$ is an imaginary element of $\OO$. Putting $x=1$ in the last equation we obtain $A1=B1$. Similarly, with $x=1$ in Equation \eqref{condition2}, we obtain $ By = -y\! \cdot \! \overline{(A1)}$. Denoting $-\overline{(A1)}=-\overline{(B1)}  \in \mathrm{Im}(\mathbb{O})$ by $c$, we have shown
\begin{equation*}
    Ax = x \! \cdot \! c \; , \quad By = y  \! \cdot \! c\,.
\end{equation*}
This turns Equation \eqref{condition2} into

 \begin{equation*}
     (x \! \cdot \! c) \! \cdot \! \Bar{y} = x \! \cdot \! (c \! \cdot \! \Bar{y}) \, ,
 \end{equation*}
 which must hold true for all $x$ and $y$. But this is possible if and only if $c \in \mathrm{Re} (\OO)$. Since we already showed that $c$ is purely imaginary, this  implies that $c=0$, and consequently that $u=v=0$, which is in contradiction with our assumption.  \hfill $\blacksquare$
\end{proof}

\begin{deff}
    Let $\cF$ be a singular foliation on a Riemannian manifold.
    \begin{itemize}
        \item $\mathcal{F}$ is called \index{homogeneous}\emph{homogeneous} if there exists a Lie group of isometries $G\subset \mathrm{Isom(M,g)}$, such that the leaves of $\cF$ are generated as the orbits of the action.

        \item $\mathcal{F}$ is called \index{locally homogeneous}\emph{locally homogeneous} if for every point $q\in M$, there exists an open subset $U$ such that $\iota_U^{-1}\cF$ is homogeneous on $(U,g_U)$.
    \end{itemize}
    Otherwise, $\cF$ is called (locally) non-homogeneous.
\end{deff}

\vskip 2mm\noindent Since the fundamental vector fields of isometries of $(\R^n,g_{st})$ which leave the origin fixed are always linear, we have the following immediate corollary  of Lemma \ref{nolin}:
\begin{cor}\label{orthlin}
    The singular foliation $\cF_{OH}$ on $\OO^2\cong \R^{16}$, equipped with the standard Riemannian metric, is both non-homogeneous and locally non-homogeneous at the origin.
    
\end{cor}

\begin{rem}
    Corollary \ref{orthlin}  can be used also as an alternative proof of the classical result that the \index{octonionic Hopf fibration}\emph{octonionic Hopf fibration} of $\mathrm{S}^{15}$ is non-homogeneous, which was shown in \cite{L93} by examining all isometric Lie  group actions on $\R^{16}$, and (implicitly) \cite{GWZ86} by different methods. 
  \end{rem}
      \begin{rem}
    The local non-homogeneity of the singular leaf decomposition $\cL_{OH}$ around the origin has been discussed in \cite{MR19}. 
    \end{rem}

\vskip 2mm\noindent Finally, we improve the previous classical results, and show that local non-homogeneity does not hold even up to Hausdorff Morita equivalence.

\begin{theorem}\label{nonhom}
    Let $\cF_0$ be \emph{any} singular foliation on $\OO^2$ having $\cL_{OH}$ as its leaf decomposition. Then $(\OO^2,\cF_0)$ is not Hausdorff Morita equivalent to any locally homogeneous singular foliation $\cF$ on some Riemannian manifold $(M,g)$. 
\end{theorem}

\vskip 2mm\noindent \begin{proof}
    Assume that $(\OO^2,\cF_0)\sim_{ME}(M,\cF)$ for some locally homogeneous singular foliation $\cF$ on a Riemannian manifold $(M,g)$. Using the first part of Theorem \ref{hme}, the origin in $\OO^2$, as the zero-dimensional  leaf of $\cF_0$, corresponds to a leaf $L_q^\cF\subset M$ of $\cF$ for some $q\in M$. Denote the orthogonal complement of $T_qL_q^\cF$ in $T_qM$ by $\nu_q:=(T_qL_q^\cF)^\perp$. We define the slice $S_q\subset M$ at $q\in L_q^\cF$ 
    as the image of the exponential map $\mathrm{exp}_q\colon \nu_q^\epsilon\to M$, where $\nu_q^\epsilon$ stands for the vectors of length smaller than some $\epsilon>0$, chosen small enough such that $\nu_q^\epsilon$ lies inside the domain of definition of $\mathrm{exp}_q$. 
    
    \vskip 2mm\noindent The singular foliation $\iota_{S_q}^{-1}\mathcal{F}$ (see Example \ref{inclusionSF} and Proposition \ref{defpullback}) is generated by vector fields in $\cF$ tangent to the slice $S_q$. The second part of Theorem \ref{hme} then implies that the foliated manifold $(S_q,\iota_{S_q}^{-1}\cF)$ is isomorphic to a neighborhood of the origin in the foliated manifold $(\OO^2,\cF_0)$. In particular, $(S_q,\iota_{S_q}^{-1}\cF)$ has a single leaf of dimension $0$ and all the other leaves are diffeomorphic to $7$-spheres.
    
    \vskip 2mm\noindent On the other hand, by assumption there exists an open subset $U\subset M$ containing $q$, such that $\iota_{S_q}^{-1}\cF$ is induced by a Lie group $G\subset \mathrm{Isom}(U,g_U)$, acting on $U$ by isometries. Assuming that $\epsilon$ is chosen small enough to have $S_q = \exp_q(\nu_q^\epsilon)\subset U$, we claim that the stabilizer $H_q$ of $G$ at $q$ acts orthogonally on $S_q$ and determines a singular foliation $\cF'$. More precisely, since $H_q$ acts on $S_q$ by isometries and for every $h\in H_q\subset G$, we have
    \begin{equation*}
        h\!\cdot\!\mathrm{exp}_q(v)=\mathrm{exp}_q(\exd_q L_h(v))\,,
    \end{equation*}
    we can define the left Lie group action $H_q\times \nu_q^\epsilon\to\nu_q^\epsilon$ by the mapping $(h,v)\to \exd_q L_h(v)$. Note that $H_q$ acts linearly and since 
    \begin{equation*}
        g(\exd_q L_h(v),\exd_q L_h(v))=L_h^*g(v,v)=g(v,v)\,,
    \end{equation*}
    it preserves the norm. Consequently, it defines an orthogonal left Lie group action on $\OO^2$. But according to \cite{L93}, an orthogonal Lie group action on $(\R^{16},g_{st})$ cannot induce $7$- dimensional leaves. This contradicts what we found above.\hfill $\blacksquare$
    \end{proof}

\section{\texorpdfstring{$\cL_{OH}$}{LOH} and singular Riemannian foliations}
\label{sec:SRF}

Throughout this section, $M$ is a smooth, polynomial or real analytic manifold and $U\subset M$ an open subset. $\cO\colon U\mapsto \cO(U)$ stands for the sheaf of smooth, polynomial or real analytic functions. We denote by $\mathfrak{X}(U)$ the $\cO(U)$-module of vector fields on $U$ and the sheaf of vector fields by $\mathfrak{X}:U\mapsto \mathfrak{X}(U)$.

\subsection{Singular Riemannian foliations}

\begin{deff}
    A subsheaf $\cF\colon U\mapsto \cF(U)$ of $\cO$-modules on $M$ is called \emph{locally finitely generated}, if for every point $q\in M$, there exists a neighborhood $U$ containing $q$ and finitely many sections $X_1,\ldots,X_N\in \cF(U)$, such that for every open subset $V\subset U$ we have
    \begin{equation*}
        \cF(V)=\la X_1\vert_V,\ldots,X_N\vert_V\ra_{\cO(V)}
    \end{equation*}
\end{deff}

\begin{deff}\label{sheafSF}
    A singular foliations---SF for short--- on an $\cO$-manifold $M$ is a subsheaf $\cF\subset\mathfrak{X}$ of the sheaf of vector fields, which is locally finitely generated and involutive, i.e.\ the $\cO(U)$-module of vector fields $\cF(U)$ is closed under the Lie bracket of vector fields. The pair $(M,\cF)$ is referred to as a foliated manifold.
\end{deff}

\vskip 2mm \noindent Theorem \ref{hermann} stays valid and implies the partition of the foliated manifold $(M,\cF)$ into leaves. Moreover, all definitions and properties mentioned in Section \ref{SFbasic} can be similarly defined and verified for Definition \ref{sheafSF} of SFs.

\begin{rem}
    The two definitions coincide in the smooth setting. More precisely, if $\cO$ is the sheaf of smooth functions on the smooth manifolds $M$, there is a one-to-one correspondence between locally finitely generated subsheaves of the sheaf of vector fields and the $C^\infty(M)$-submodules of the compactly supported vector fields \cite{GZ19}.
\end{rem} 

\vskip 2mm \noindent There are two approaches to define singular Riemannian foliations---SRF for short---on Riemannian manifolds. The first approach due to Molino \cite{M98} defines a compatibility condition between the leaf decomposition induced by $\cF$ and the Riemannian structure. This approach is summerized in the following definition.

\begin{deff}\cite{NS24}\label{gsrf}
    Let $\cF$ be a singular foliation in the Riemannian manifold $(M,g)$. The triple $(M,g,\cF)$ defines a \emph{geometric SRF}, if every geodesic starting perpendicular to a leaf, stays perpendicular to all the leaves it meets.
\end{deff}

\vskip 2mm \noindent As it is clear in this approach, it only concerns the leaf decomposition of $\cF$, and not its additional algebraic properties. To capture these additional data, the following alternative definition is proposed in \cite{NS24}.

\begin{deff}\cite{NS24}\label{msrf}
    Let $\cF$ be a singular foliation in the Riemannian manifold $(M,g)$. The triple $(M,g,\cF)$ defines a \emph{module SRF}, if for every open subset $U\subset M$ and $X\in \cF(U)$ we have
    \begin{equation}\label{srfeqn}
    \cL_Xg\in \Omega^1(M)\odot g_\flat(\cF(U))\,.
    \end{equation}
    Here, $\odot$ stands for the symmetric inner product of $1$-forms and $g_\flat: TM\to T^*M$ is the musical isomorphism given by $(q,v)\mapsto g_q(v,\cdot)$ for all $(q,v)\in TM$.
\end{deff}

\vskip 2mm \noindent In particular, for an open subset $U$ with finitely many generators $X_1,\ldots,X_N\in \cF(U)$, Equation \eqref{srfeqn} is satisfied if and only if there exist $1$-forms $\omega_a^b\in \Omega^1(U)$ such that
\begin{equation}\label{srfeqn2}
    \cL_{X_a}g=\omega_a^b\odot g(X_b,\cdot)\,.
\end{equation}

\begin{rem}
    Definitions \ref{gsrf} and \ref{msrf} are introduced in \cite{NS24} for the definition of SFs as $C^\infty(M)$-submodules of compactly supported vector fields, but it is observed that it can be directly generalized to the definition of SFs as subsheaves.
\end{rem}

\vskip 2mm \noindent In \cite{NS24}, Propositon $3.4$ verifies that both definitions agree on regular foliations, and Proposition $3.2$ states that every module SRF is also a geometric SRF. However, the converse does not hold true in general. 

\begin{example}\label{nonexrot}
    Consider the SF $\mathcal{F}\subset \mathfrak{X}$ generated by a single vector field $V:=(x^2+y^2)(x\partial_y-y\partial_x)$ on $M=\R^2$ equipped with the standard metric $\exd s^2$. The leaves are circles centered at the origin, which defines a geometric SRF, but it does not satisfy Equation \eqref{srfeqn2}. More precisely, a simple calculation implies that 
    \begin{equation*}
        \cL_V\exd s^2=4\left[\frac{x\exd x+y\exd y}{x^2+y^2}\right]\odot (\exd s^2)_\flat (V),
    \end{equation*}
    on $\R^2\setminus {(0,0)}$. Evidently, the $1$-form $\frac{x\exd x+y\exd y}{x^2+y^2}$ fails to have a smooth extension at the origin.
\end{example}

\vskip 2mm \noindent Despite the fact that the SF introduced in Example \ref{nonexrot} does not define a module SRF on $\R^2$, it is possible to find a module SRF with the same leaf decomposition. In fact, the SF $\cF_0$ generated by the vector field $V_0=x\partial_y-y\partial_x$ is a Killing vector field for the standard metric $\exd s^2$, and Equation \eqref{srfeqn} is obviously satisfied.

\subsection{$\cF_{OH}$, a counter-example}

\vskip 2mm \noindent \textbf{Question:} Let $(M,g,\cF)$ be a geometric singular Riemannian foliation. Is it possible to find a module singular Riemannian foliation $(M,g,\cF')$, having the same leaf decomposition as $(M,g,\cF)$?

\vskip 2mm \noindent Here, we claim that for $\cO$ being the sheaf of real analytic functions on a real analytic manifold, the singular octonionic Hopf foliation provides a counter-example.

\vskip 2mm \noindent Consider $\OO^2$ as a real analytic manifold, equipped with the standard Riemannian metric $g_{st}$. The restriction of $\cF_{OH}$ to $\mathrm{S}^{15}\subset \OO^2$ induces the octonionic Hopf fibration, which is known to be a regular Riemannian foliation. M98's homothetic transformation lemma \cite{M98} then implies that $(\OO^2, g_{st},\cF_{OH})$ itself defines a geometric singular Riemannian foliation.

\begin{theorem}
    Let $\cF$ be any singular foliation on the real analytic Riemannian manifold $(\OO^2,g_{st})$, having $\cL_{OH}$ as its leaf decomposition. Then the geometric singular Riemannian foliation $(\OO^2, g_{st},\mathcal{F})$ is not a module singular Riemannian foliation.
\end{theorem}

\begin{proof}
    Let $(\OO^2, g_{st}, \cF)$ be a module singular Riemannian foliation on the real analytic manifold $\OO^2$, with $\cL_{OH}$ as its leaf decomposition, and consider the vector field $X\in \mathcal{F}$. The Taylor expansion around the origin gives homogeneous polynomials $P_k(x,y)$ and $Q_k(x,y)$ of degree $k$, such that

    \begin{equation*}
        X(x,y)=\la P_k(x,y),\frac{\partial}{\partial x}\ra+\la Q_k(x,y),\frac{\partial}{\partial y}\ra\,.
    \end{equation*}

    \vskip 2mm \noindent Denote the homogeneous part of degree $k$ in $X$ by $X_k$. Since $X$ is tangent to the leaves in $\cL_{OH}$, we have $X(0,0)=0$ which translates into

    \begin{equation*}
        P_0(0,0)=Q_0(x,y)\equiv 0\,.
    \end{equation*}
    
    \vskip 2mm \noindent In addition,  Equation \eqref{condition} implies that 
    \begin{equation*}
        \sum_{k=1}^\infty P_k(x,y)\!\cdot\!\overline{y}+ x\!\cdot\!\overline{Q_k(x,y)}=0
    \end{equation*}
    and in particular 
    \begin{equation*}
        \la P_1(x,y),\frac{\partial}{\partial x}\ra+\la Q_1(x,y),\frac{\partial}{\partial y}\ra=0\,.
    \end{equation*}
    
    \vskip 2mm \noindent Similarly, we can show that the linear vector field $X_1$ satisfies Equation \eqref{orthogonal}, and Lemma \ref{lemchar} implies that it is tangent to the leaves of $\cL_{OH}$. But according to Lemma \ref{nolin}, there is no there is no non-zero linear vector field tangent to the leaves in $\cL_{OH}$, showing that $X_0=X_1=0$. Consequently we have $X=\sum_{k=2}^\infty X_k$

    \vskip 2mm \noindent As a direct consequence of the result above, every element in the space $\Omega^1(\OO^2)\odot(g_{st})_\flat(\cF)$ vanishes at least quadratically at the origin.

    \vskip 2mm \noindent In the other hand, we have

    \begin{equation*}
        \cL_{X}g_{st}=\sum_{k=2}^\infty \cL_{X_k}g_{st}\,.
    \end{equation*}

    \vskip 2mm \noindent For $k=2$, the homogeneous part $\cL_{X_2}g_{st}$ is a symmetric $2$-tensor of degree $1$ in $x$ and $y$. This term cannot be an element of $\Omega^1(\OO^2)\odot(g_{st})_\flat(\cF)$, and $X_2$ is not a Killing vector field, consequently we obtain $X_2=0$.

    \vskip 2mm \noindent Recursively, we can prove that $X_k=0$ for all $k$, which implies $X=0$. This completes the proof.\hfill $\blacksquare$ 
\end{proof}

\section{Universal Lie \texorpdfstring{$\infty$}{infty}-algebroid of \texorpdfstring{$\cF_{OH}$}{FOH}}\label{lieinf}

 \vskip 2mm \noindent In \cite{LGLS20}, it is proven that for every singular foliation which admits a \index{geometric resolution}\emph{geometric resolution}, one can associate a \index{Lie $\infty$-algebroid}\emph{Lie $\infty$-algebroid} inducing it. (The necessary notions will be recalled in Section \ref{lieinfbackground}). This association turns out to be unique up to homotopy and leads to invariants of the singular foliation. In Section \ref{lie3}, we construct a Lie $3$-algebroid that represents the universal Lie $\infty$-algebroid of $\cF_{OH}$. This Lie $3$-algebroid will be used then in to prove that $\cG \Rightarrow \OO^2$ has the minimal dimension among all Lie groupoids having $\cL_{OH}$ as their orbits.

\subsection{Lie \texorpdfstring{$\infty$}{infty}-algebroids of singular foliations}\label{lieinfbackground}
\vskip 2mm \noindent In what follows, $M$ is a smooth or real analytic manifold, or an affine variety over $\K=\R$ or $\C$, whose sheaf of functions is denoted by \nomenclature{$\cO$}{\qquad Sheaf of rings of polynomial, real analytic or smooth functions}$\cO$. For every vector bundle $F\to M$, the space of sections $\Gamma(F)$ is viewed as a sheaf of $\cO$-modules.

\begin{deff}
    A (split) \index{Lie $\infty$-algebroid}\emph{Lie $\infty$-algebroid} $(E,(l_k)_{k\geq 1},\rho)$ consists of a positively-graded vector bundle $E = \bigoplus_{i \geq 0} E_{-i}$ over a manifold $M$, equipped with:

    \begin{enumerate}
        \item A family of graded skew-symmetric and multilinear maps $l_k \colon \wedge^k \Gamma(E) \to \Gamma(E)$ of degree $2 - k$, called $k$-brackets, for all integers $k \geq 1$,
        \item A vector bundle morphism $\rho \colon E_0 \to TM$, called the anchor.
    \end{enumerate}

    These must satisfy the following conditions:
    \begin{itemize}
        \item For $k \neq 2$, the $k$-brackets $l_k$ are $\mathcal{O}$-linear,
        \item The $2$-bracket $l_2$ is $\mathcal{O}$-linear, except when at least one of the entries is of degree zero: for sections $x \in \Gamma(E_0)$, $y \in \Gamma(E)$, and $f \in \mathcal{O}$, the $2$-bracket satisfies
        \begin{equation*}
            l_2(x, fy) = f \,l_2(x, y) + (\rho(x) \cdot f) \, y;
        \end{equation*}
        \item For every $x \in \Gamma(E_{-1})$, one has $\rho(l_1(x)) = 0$;
        \item The $k$-brackets satisfy the \index{higher Jacobi identities}\emph{higher Jacobi identities}: for every positive integer $n$ and sections $x_1, \ldots, x_n \in \Gamma(E)$,
        \begin{equation*}
            \sum_{i=1}^n (-1)^{i(n-i)} \sum_{\sigma \in \mathrm{Un}(i,n-i)} \epsilon(\sigma) l_{n-i+1}(l_i(x_{\sigma(1)}, \ldots, x_{\sigma(i)}), x_{\sigma(i+1)}, \ldots, x_{\sigma(n)}) = 0,
        \end{equation*}
        where $\mathrm{Un}(i, n-i)$ stands for the set of $(i, n-i)$-unshuffles, and $\epsilon(\sigma)$ is the signature of the permutation $\sigma$ given by
        \begin{equation*}
            x_{\sigma(1)} \wedge \ldots \wedge x_{\sigma(n)} = \epsilon(\sigma) x_1 \wedge \ldots \wedge x_n.
        \end{equation*}
    \end{itemize}
    
    A Lie $n$-algebroid is a Lie $\infty$-algebroid with $E_{-i}=0$ for $i\geq n$.
\end{deff}

\vskip 2mm \noindent 
The $\cO$-linearity of  $l_k$ with $k\neq 2$ implies that for these $k$ we have $l_k \colon \wedge^k E \to E$. (We do not distinguish between the vector bundle morphisms and the induced map on sections). 

\begin{example}\label{exdg}
    A Lie $\infty$-algebroid $(E, (l_k)_{k \geq 1}, \rho)$ with $l_k = 0$ for all $k \geq 3$ is called a \index{dg-Lie algebroid}\emph{dg-Lie algebroid}. Denoting the $1$-bracket by $\exd$ and the $2$-bracket by $[\cdot,\cdot]$, the higher Jacobi identities become:
    \begin{align}
        &\exd \circ \exd = 0, \label{ddifferential}\\
        &\exd [x,y] = [\exd x, y] + (-1)^{|x|} [x, \exd y], \label{ddistr}\\
        &(-1)^{|x||z|} [x, [y, z]] + (-1)^{|y||x|} [y, [z, x]] + (-1)^{|z||y|} [z, [x, y]] = 0, \label{jacobi}
    \end{align}
    for all homogeneous sections $x, y,$ and $z$. Here, $|x|$ denotes the degree of the homogeneous section $x$.
\end{example}

\vskip 2mm \noindent 
In a Lie $\infty$-algebroid, writing $\exd^{(i)} := l_1 \vert_{E_{-i}}$, the higher Jacobi identities implies that the sequence
\begin{equation*}
    \begin{tikzcd}
        \cdots \arrow[r] & E_{-3} \arrow[r, "\exd^{(3)}"] & E_{-2} \arrow[r, "\exd^{(2)}"] & E_{-1} \arrow[r, "\exd^{(1)}"] & E_{0} \arrow[r, "\rho"] & TM
    \end{tikzcd}
\end{equation*}
is a chain complex, called the \index{linear part}\emph{linear part} of the Lie $\infty$-algebroid.

\begin{deff}\label{geomressf}
    Let $\cF$ be a singular foliation on $M$. A \index{geometric resolution of $\cF$}\emph{geometric resolution of $\cF$} is a chain complex
    \begin{equation*}
    \begin{tikzcd}
        \cdots \arrow[r] & E_{-3} \arrow[r, "\exd^{(3)}"] & E_{-2} \arrow[r, "\exd^{(2)}"] & E_{-1} \arrow[r, "\exd^{(1)}"] & E_{0} \arrow[r, "\rho"] & TM
    \end{tikzcd}
\end{equation*}
with $\rho(\Gamma(E_0))=\cF$, such that for every open subset $U\subset M$ the chain complex
\begin{equation*}
    \begin{tikzcd}
        \cdots \arrow[r] & \Gamma_U(E_{-3}) \arrow[r, "\exd^{(3)}"] & \Gamma_U(E_{-2}) \arrow[r, "\exd^{(2)}"] & \Gamma_U(E_{-1}) \arrow[r, "\exd^{(1)}"] & \Gamma_U(E_{0}) \arrow[r, "\rho"] & \cF(U)\arrow[r, ""] &0
    \end{tikzcd}
\end{equation*}
is an exact sequence of $\cO(U)$-modules. The geometric resolution $(E,\exd,\rho)$ is said to be \index{minimal at $q\in M$}\emph{minimal at $q\in M$}, if $\exd_q^{(i)}\colon E_{-i}\vert_q\to E_{(i-1)\vert_q}$ vanishes for all $i\geq 1$.
\end{deff}

\begin{lemma}[\cite{LGLS20}]\label{geomrescorres}
    Let $\cF$ be a singular foliation on an $\cO$-manifold $M$.\\
    For $\cO$ being the sheaf of smooth functions on $M$, geometric resolutions $\cF$ are in one-to-one correspondence with resolutions of the $\cO$-module $\cF$ by locally finitely generated projectvice $\cO$-modules.\\
    For $\cO$ being the sheaf of polynomial or real analytic functions on $M$, geometric resolutions are in one-to-one correspondence with resolutions of $\cF$ by finitely generated free $\cO$-modules.
\end{lemma}

\vskip 2mm \noindent The flatness theorems of Malgrange \cite{T68} imply the following proposition on the transition between different choices of $\cO$ on $\R^n$ in the study of geometric resolutions.

\begin{prop}[\cite{LGLR22}]\label{poltosmooth}
    Let $\cF$ be a singular foliation on the affine variety $\R^n$. A geometric resolution of $F$ for $\cO$ being the sheaf of polynomial functions, is also a geometric resolution for $\cO$ being the sheaf of smooth functions. 
\end{prop}

\begin{deff}\label{lieinfsf}
    Let $\cF$ be a singular foliation on $M$. A \index{universal Lie $\infty$-algebroid of $\cF$}\emph{universal Lie $\infty$-algebroid of $\cF$} is a Lie $\infty$-algebroid $(E,(l_k)_{k\geq 1},\rho)$, such that its linear part is a geometric resolution of $\cF$. If $E_{-i}=0$ for $i\geq n$, it is called a universal Lie $n$-algebroid.
\end{deff}

\vskip 2mm \noindent As it is clear from Definition \ref{lieinfsf}, for a singular foliation to be induced by a universal Lie $\infty$-algebroid, it is required to admit a geometric resolution. There are examples of singular foliations which do not admit a geometric resolution, see Example $3.38$ in \cite{LGLS20}. However, if $\cO$ is the sheaf of polynomial or real analytic functions, Hilbert’s
syzygy theorem ensures the existence of a geometric resolution in a neighborhood of every point.

\begin{theorem}[\cite{LGLS20}]
    Let $\cF$ be a singular foliation on $M$, admitting a geometric resolution $(E, \exd, \rho)$. There exist a universal Lie $\infty$-algebroid, having $(E, \exd, \rho)$ as its linear part.
\end{theorem}

\vskip 2mm \noindent  Such a universal Lie $\infty$-algebroid is unique up to some precise notion of homotopies; for details on this and about why it is universal, see \cite{LGLS20}. 


\vskip 2mm \noindent Let us recall that for a foliated manifold $(M,\cF)$ the minimal number of locally generating vector fields for $\cF$ around $q\in M$ is equal to the dimension of the fiber $\cF_q$. 

\begin{prop}[\cite{LGLS20}]\label{minrank}
    Let $(E, (l_k)_{k \geq 1}, \rho)$ be a universal Lie $\infty$-algebroid of a singular foliation $\cF$ on $M$. If the linear part $(E, l_1, \rho)$ is minimal 
    at $q\in M$, then $\mathrm{rank}(E_{0})=\dim(\cF_q)=:r$. In particular, the rank of every Lie algebroid inducing $\cF$ in a neighborhood of $q$ is at least $r$.
\end{prop}

\subsection{The universal Lie \texorpdfstring{$3$}{3}-algebroid of of \texorpdfstring{$\cF_{OH}$}{FOH}}\label{lie3}
We start the construction of the Lie $3$-algbroid by choosing a geometric resolution for $\cF_{OH}$. Let us first view $\cF_{OH}$ as a sheaf of $\cO$-modules over the affine variety $\OO^2\cong\R^{16}$, where $\cO$ is the sheaf of polynomials functions. This is possible since $\cF_{OH}$ is generated by polynomial vector fields.  In this setting, according to Lemma \ref{geomrescorres}, it suffices to find a free resolution of the module $\cF_{OH}$. This can be done using Eisenbaud's Macaulay2, as explained in Appendix \ref{appa}. The result is as follows:

\vskip 2mm \noindent The graded vector bundle of the geometric resolution is given by the trivial vector bundles $E_{0}$, $E_{-1}$ and $E_{-2}$ over $\OO^2$, where 
\begin{align*}
    E_{0}:=\underline{\OO^2},\qquad
    E_{-1}:=\underline{\R\oplus\OO\oplus\R},\qquad
    E_{-2}:=\underline{\R}\,.
\end{align*}

\vskip 2mm \noindent  The vector bundle morphisms $\rho$, $\exd^{(1)}$ and $\exd^{(2)}$ in the sequence 
\begin{equation}
    \begin{tikzcd}
0\arrow[r] & E_{-2} \arrow[r, "\exd^{(2)}"] & E_{-1} \arrow[r, "\exd^{(1)}"] & E_{0} \arrow[r, "\rho"] & T\OO^2\cong \underline{\OO^2}
\end{tikzcd}
\end{equation}
\vskip 2mm \noindent are given by the following evaluations on the constant sections $\begin{pmatrix}
    u \\
    v
\end{pmatrix}\in \Gamma(E_{0})$, $\begin{pmatrix}
    \mu \\
    a   \\
    \nu
\end{pmatrix}\in \Gamma(E_{-1})$ and $t\in \Gamma(E_{-2})$:
\begin{align}\nonumber
\rho\begin{pmatrix}
u \\
v
\end{pmatrix}&:=\begin{pmatrix}
\|x\|^2u+(x\!\cdot\!\overline{y})\!\cdot\!v-(\la x,u\ra+\la y,v\ra)x \\
\|y\|^2v+(y\!\cdot\!\overline{x})\!\cdot\!u-(\la x,u\ra+\la y,v\ra)y
\end{pmatrix}\,,\\
    \exd^{(1)}\begin{pmatrix}
    \mu \\
    a   \\
    \nu 
\end{pmatrix}&:=\begin{pmatrix}
    \mu x+a\!\cdot\!y \\
    \nu y+\overline{a}\!\cdot\!x
\end{pmatrix}\,,\\
    \exd^{(2)}(t)&:=\begin{pmatrix}
    -\|y\|^2t \\
    (x\!\cdot\!\overline{y})t   \\
    -\|x\|^2t
\end{pmatrix}\,.\nonumber
\end{align}

\begin{lemma}
    The triple $(E\!=\!\oplus_{i=0}^2E_{-i}, \exd, \rho)$ is a geometric resolution of $\cF_{OH}$ on the smooth manifold $\OO^2\cong\R^{16}$. Moreover, this geometric resolution is minimal at the origin.
\end{lemma}

\begin{proof}
Since $(E\!=\!\oplus_{i=0}^2E_{-i}, \exd, \rho)$ is a geometric resolution of $\cF_{OH}$ as a sheaf of modules over the ring of polynomials, Proposition \ref{poltosmooth} implies that it is also a geometric resolution for $\cF_{OH}$ when $\cO$ is the sheaf of smooth functions on $\R^n$. Clearly, both $\exd^{(1)}$ and $\exd^{(2)}$ vanish at the origin, showing that the geometric resolution is minimal at this point.\qquad$\blacksquare$
\end{proof}

\vskip 2mm \noindent We construct a degree $2$-bracket $[\cdot,\cdot]\colon \Gamma(E)\wedge \Gamma(E) \to \Gamma(E)$ of degree $0$ as follows:
\begin{align}
[\begin{pmatrix}u\\v\end{pmatrix},\begin{pmatrix}u'\\v'\end{pmatrix}]&:=(\la x,u\ra+\la y,v\ra)\begin{pmatrix}u'\\v'\end{pmatrix}-(\la x,u'\ra+\la y,v'\ra)\begin{pmatrix}u\\v\end{pmatrix}\,,\nonumber\\\,
[\begin{pmatrix}u\\v\end{pmatrix},\begin{pmatrix}\mu\\a\\\nu\end{pmatrix}]&:=\begin{pmatrix}-2\la y,\overline{a}\!\cdot\!u\ra+2\la y,v\ra\mu\\x\!\cdot\!(\overline{u}\!\cdot\!a)+(a\!\cdot\!v)\!\cdot\!\overline{y}-\mu (x\!\cdot\!\overline{v})-\nu (u\!\cdot\!\overline{y})\\-2\la x,a\!\cdot\!v\ra+2\la x,u\ra\nu\end{pmatrix}\,,\nonumber\\\,
[\begin{pmatrix}u\\v\end{pmatrix},t]&:=2(\la x,u\ra+\la y,v\ra)t\,,\nonumber\\\,
[\begin{pmatrix}\mu\\a\\\nu\end{pmatrix},\begin{pmatrix}\mu'\\a'\\\nu'\end{pmatrix}]&:=4\la a,a'\ra-2\mu\nu'-2\mu'\nu\,,
\end{align}
\vskip 2mm \noindent and extend it to all sections of $E$ using Leibniz rule. We define $l_k=0$ for all $k\geq3$. 

\begin{prop}\label{thmlie3}
    The triple $(E, (l_k)_{k \geq 1}, \rho)$ defines a Lie $3$-algebroid. It is a universal Lie $3$-algebroid of $\cF_{OH}$, whose linear part is a geometric resolution minimal at the origin.
\end{prop}

\begin{proof}
    \noindent Since $l_k=0$ for $k\geq 3$, we have to show that $(E, (l_k)_{k \geq 1}, \rho)$ is a dg-Lie algebroid. Note that as $(E, \exd, \rho)$ is a geometric resolution, we have $\rho \circ \exd=0$ and $\exd \circ \exd=0$. It suffices to verify Equations \eqref{ddistr} and $\eqref{jacobi}$ in Example $\eqref{exdg}$. 
    
    {\bf Step $1$. Verifying Equation \eqref{ddistr}:} For constant sections $\renewcommand{\arraystretch}{0.7}\begin{pmatrix}u\\v\end{pmatrix}\in \Gamma(E_0)$ and $\renewcommand{\arraystretch}{0.7}\begin{pmatrix}\mu\\a\\\nu\end{pmatrix}\in \Gamma(E_{-1})$, we have

    \begin{align*}
        &\exd^{(1)}([\begin{pmatrix}u\\v\end{pmatrix},\begin{pmatrix}\mu\\a\\\nu\end{pmatrix}])=\exd^{(1)}\begin{pmatrix}-2\la y,\overline{a}\!\cdot\!u\ra+2\la y,v\ra\mu\\x\!\cdot\!(\overline{u}\!\cdot\!a)+(a\!\cdot\!v)\!\cdot\!\overline{y}-\mu (x\!\cdot\!\overline{v})-\nu (u\!\cdot\!\overline{y})\\-2\la x,a\!\cdot\!v\ra+2\la x,u\ra\nu\end{pmatrix}\\
        &=\begin{pmatrix} -2\la y,\overline{a}\!\cdot\!u\ra x+2\la y,v\ra\mu x+(x\!\cdot\!(\overline{u}\!\cdot\!a))\!\cdot\!y+\|y\|^2a\!\cdot\!v-\mu (x\!\cdot\!\overline{v})\!\cdot\!y-\nu\|y\|^2 u\\-2\la x,a\!\cdot\!v\ra y+2\la x,u\ra\nu y+(y\!\cdot\!(\overline{v}\!\cdot\!\overline{a}))\!\cdot\!x+\|x\|^2\overline{a}\!\cdot\!u-\nu (y\!\cdot\!\overline{u})\!\cdot\! x-\mu\|x\|^2 v\end{pmatrix}\\
        &=\begin{pmatrix} -(x\!\cdot\!\overline{y})\!\cdot\!(\overline{a}\!\cdot\!u)+\|y\|^2a\!\cdot\!v+\mu (x\!\cdot\!\overline{y})\!\cdot\!v-\nu\|y\|^2 u\\-(y\!\cdot\!\overline{x})\!\cdot\!(a\!\cdot\!v)+\|x\|^2\overline{a}\!\cdot\!u+\nu (y\!\cdot\!\overline{x})\!\cdot\! u-\mu\|x\|^2 v\end{pmatrix}\,,
    \end{align*}
    where we used Equation \eqref{semiasseqn2} to obtain the second equality, and

    \begin{align*}
        &[\begin{pmatrix}u\\v\end{pmatrix},\exd^{(1)}\begin{pmatrix}\mu\\a\\\nu\end{pmatrix}]=[\begin{pmatrix}u\\v\end{pmatrix},\begin{pmatrix} \mu x+a\!\cdot\!y\\\nu y+\overline{a}\!\cdot\!x\end{pmatrix}]\\
        &=\begin{pmatrix}\mu(\|x\|^2u+(x\!\cdot\!\overline{y})\!\cdot\!v-(\la x,u\ra+\la y,v\ra)x)+a\!\cdot\!(\|y\|^2v+(y\!\cdot\!\overline{x})\!\cdot\!u-(\la x,u\ra+\la y,v\ra)y)\\\nu(\|y\|^2v+(y\!\cdot\!\overline{x})\!\cdot\!u-(\la x,u\ra+\la y,v\ra)y)+\overline{a}\!\cdot\!(\|x\|^2u+(x\!\cdot\!\overline{y})\!\cdot\!v-(\la x,u\ra+\la y,v\ra)x)\end{pmatrix}\\
        &+(\la x,u\ra+\la y,v\ra)\begin{pmatrix}\mu x+a\!\cdot\!y\\\nu y+\overline{a}\!\cdot\!x\end{pmatrix}-(\la x, \mu x+a\!\cdot\!y\ra+\la y,\nu y+\overline{a}\!\cdot\!x\ra)\begin{pmatrix}u\\v\end{pmatrix}\,.
    \end{align*}
    After cancellations and using Equations \eqref{switch1} and \eqref{switch2} in the last term, it becomes equal to
    \begin{align*}
        =\begin{pmatrix}a\!\cdot\!((y\!\cdot\!\overline{x})\!\cdot\!u)+\|y\|^2a\!\cdot\!v+\mu(x\!\cdot\!\overline{y})\!\cdot\!v-\nu\|y\|^2u\\-\overline{a}\!\cdot\!((x\!\cdot\!\overline{y})\!\cdot\!v)+\|x\|^2\overline{a}\!\cdot\!u+\nu (y\!\cdot\!\overline{x})\!\cdot\! u-\mu\|x\|^2 v\end{pmatrix}-2\la a,x\!\cdot\!\overline{y}\ra\begin{pmatrix}u\\v\end{pmatrix}\,,
    \end{align*}
    which together with Equation \eqref{semiasseqn2} gives \[\exd^{(1)}([\begin{pmatrix}u\\v\end{pmatrix},\begin{pmatrix}\mu\\a\\\nu\end{pmatrix}])=[\begin{pmatrix}u\\v\end{pmatrix},\exd^{(1)}\begin{pmatrix}\mu\\a\\\nu\end{pmatrix}]\,.\]
    
\vskip 2mm \noindent For global constant sections $\renewcommand{\arraystretch}{0.7}\begin{pmatrix}u\\v\end{pmatrix}\in \Gamma(E_0)$ and $t\in \Gamma(E_{-2})$ one has
\begin{align*}
    \exd^{(2)}([\begin{pmatrix}u\\v\end{pmatrix},t])=2(\la x,u\ra+\la y,v\ra)\exd^{(t)}(t)=2(\la x,u\ra+\la y,v\ra)\begin{pmatrix}-\|y\|^2t\\(x\!\cdot\!\overline{y})t\\-\|x\|^2t\end{pmatrix}\,.
\end{align*}
On the other hand, since by Lemma \ref{sameslop} and Corollary \ref{corxy} the functions $\|x\|^2,\|y\|^2$ and $x\!\cdot\!\overline{y}$ are constant along orbits of $\cG\Rightarrow\OO^2$, we have
\begin{equation*}
    \rho\begin{pmatrix}u\\v\end{pmatrix}\cdot \begin{pmatrix}-\|y\|^2t\\(x\!\cdot\!\overline{y})t\\-\|x\|^2t\end{pmatrix}=\begin{pmatrix}0\\0\\0\end{pmatrix}\,,
\end{equation*}
which gives
\begin{align*}
    [\begin{pmatrix}u\\v\end{pmatrix},\exd^{(2)}t]&=[\begin{pmatrix}u\\v\end{pmatrix},\begin{pmatrix}-\|y\|^2t\\(x\!\cdot\!\overline{y})t\\-\|x\|^2t\end{pmatrix}]\\
    &=\begin{pmatrix}-2\la y,(y\!\cdot\!\overline{x})\!\cdot\!u\ra t-2\|y\|^2\la y,v\ra t\\x\!\cdot\!(\overline{u}\!\cdot\!(x\!\cdot\!\overline{y}))t+((x\!\cdot\!\overline{y})\!\cdot\!v)\!\cdot\!\overline{y}t+\|y\|^2x\!\cdot\!\overline{v}t+\|x\|^2u\!\cdot\!\overline{y} t\\-2\la x,(x\!\cdot\!\overline{y})\!\cdot\!v\ra t-2\|x\|^2\la x,u\ra t\end{pmatrix}
\end{align*}
Using Equation \eqref{switch1}, the first and the third components become $2(\la x,u\ra+\la y,v\ra)(-\|y\|^2t)$ and $2(\la x,u\ra+\la y,v\ra)(-\|x\|^2t)$, respectively. After Moufang identities \eqref{moufang2} and \eqref{moufang3}, and using Equation \eqref{semiasseqn2}, the second component equals $2(\la x,u\ra+\la y,v\ra)((x\!\cdot\!\overline{y})t)$. These together give
\[\exd^{(2)}([\begin{pmatrix}u\\v\end{pmatrix},t])=[\begin{pmatrix}u\\v\end{pmatrix},\exd^{(2)}t]\,.\]

\vskip 2mm \noindent Finally, for constant sections $\renewcommand{\arraystretch}{0.7}\begin{pmatrix}\mu\\a\\\nu\end{pmatrix},\renewcommand{\arraystretch}{0.7}\begin{pmatrix}\mu'\\a'\\\nu'\end{pmatrix}\in \Gamma(E_{-1})$ one has
\begin{align*}
    \exd^{(2)}([\begin{pmatrix}\mu\\a\\\nu\end{pmatrix},\begin{pmatrix}\mu'\\a'\\\nu'\end{pmatrix}])=\exd^{(2)}(4\la a,a'\ra-\mu\nu'-\mu'\nu)=\begin{pmatrix}-2\|y\|^2(2\la a,a'\ra-\mu\nu'-\mu'\nu)\\2(x\!\cdot\!\overline{y})(2\la a,a'\ra-\mu\nu'-\mu'\nu)\\-2\|x\|^2(2\la a,a'\ra-\mu\nu'-\mu'\nu)\end{pmatrix}\,.
\end{align*}
On the other hand, one has
\begin{small}
\begin{align*}
    &[\exd^{(1)}\begin{pmatrix}\mu\\a\\\nu\end{pmatrix},\begin{pmatrix}\mu'\\a'\\\nu'\end{pmatrix}]-[\begin{pmatrix}\mu\\a\\\nu\end{pmatrix},\exd^{(1)}\begin{pmatrix}\mu'\\a'\\\nu'\end{pmatrix}]\\
    &=[\exd^{(1)}\begin{pmatrix}\mu\\a\\\nu\end{pmatrix},\begin{pmatrix}\mu'\\a'\\\nu'\end{pmatrix}]+\{\begin{pmatrix}\mu\\a\\\nu\end{pmatrix}\xleftrightarrow{} \begin{pmatrix}\mu'\\a'\\\nu'\end{pmatrix}\}=[\begin{pmatrix} \mu x+a\!\cdot\!y\\\nu y+\overline{a}\!\cdot\!x\end{pmatrix},\begin{pmatrix}\mu'\\a'\\\nu'\end{pmatrix}]+\{\begin{pmatrix}\mu\\a\\\nu\end{pmatrix}\xleftrightarrow{} \begin{pmatrix}\mu'\\a'\\\nu'\end{pmatrix}\}\\
    &=\begin{pmatrix}-2\la y,\overline{a'}\!\cdot\!(\mu x+a\!\cdot\!y)\ra+2\la y,\nu y+\overline{a}\!\cdot\!x\ra\mu'\\x\!\cdot\!((\mu \overline{x}+\overline{y}\!\cdot\!\overline{a})\!\cdot\!a')+(a'\!\cdot\!(\nu y+\overline{a}\!\cdot\!x))\!\cdot\!\overline{y}-\mu' (x\!\cdot\!(\nu \overline{y}+\overline{x}\!\cdot\!a))-\nu' ((\mu x+a\!\cdot\!y)\!\cdot\!\overline{y})\\-2\la x,a'\!\cdot\!(\nu y+\overline{a}\!\cdot\!x)\ra+2\la x,\mu x+a\!\cdot\!y\ra\nu'\end{pmatrix}\\
    &+\{\begin{pmatrix}\mu\\a\\\nu\end{pmatrix}\xleftrightarrow{} \begin{pmatrix}\mu'\\a'\\\nu'\end{pmatrix}\}\\
    &=\begin{pmatrix}-2\|y\|^2(\la a,a'\ra-\mu'\nu)+2\la a,x\!\cdot\!\overline{y}\ra\mu'-2\la a',x\!\cdot\!\overline{y}\ra\mu\\(\mu\|x\|^2+\nu\|y\|^2)a'-(\mu'\|x\|^2+\nu'\|y\|^2)a+x\!\cdot\!((\overline{y}\!\cdot\!\overline{a})\!\cdot\!a')+(a'\!\cdot\!(\overline{a}\!\cdot\!x))\!\cdot\!\overline{y}-(x\!\cdot\!\overline{y})(\mu\nu'+\mu'\nu)\\-2\|x\|^2(\la a,a'\ra-\mu\nu')+2\la a,x\!\cdot\!\overline{y}\ra\nu'-2\la a',x\!\cdot\!\overline{y}\ra\mu'\end{pmatrix}\\
    &+\{\begin{pmatrix}\mu\\a\\\nu\end{pmatrix}\xleftrightarrow{} \begin{pmatrix}\mu'\\a'\\\nu'\end{pmatrix}\}\\
    &=\begin{pmatrix}-2\|y\|^2(2\la a,a'\ra-\mu\nu'-\mu'\nu)\\x\!\cdot\!((\overline{y}\!\cdot\!\overline{a})\!\cdot\!a'+(\overline{y}\!\cdot\!\overline{a'})\!\cdot\!a)+(a\!\cdot\!(\overline{a'}\!\cdot\!x)+a'\!\cdot\!(\overline{a}\!\cdot\!x))\!\cdot\!\overline{y}-(x\!\cdot\!\overline{y})(\mu\nu'+\mu'\nu)\\-2\|x\|^2(\la a,a'\ra-\mu\nu'-\mu'\nu)\end{pmatrix}\,,
    \end{align*}
    \end{small}
    \noindent where $\renewcommand{\arraystretch}{0.7}\{\begin{pmatrix}\mu\\a\\\nu\end{pmatrix}\xleftrightarrow{} \renewcommand{\arraystretch}{0.7}\begin{pmatrix}\mu'\\a'\\\nu'\end{pmatrix}\}$ stands for the terms obtained by switching of the sections. Equation \eqref{semiasseqn2} then implies
    \[\exd^{(2)}([\begin{pmatrix}\mu\\a\\\nu\end{pmatrix},\begin{pmatrix}\mu'\\a'\\\nu'\end{pmatrix}])=[\exd^{(1)}\begin{pmatrix}\mu\\a\\\nu\end{pmatrix},\begin{pmatrix}\mu'\\a'\\\nu'\end{pmatrix}]-[\begin{pmatrix}\mu\\a\\\nu\end{pmatrix},\exd^{(1)}\begin{pmatrix}\mu'\\a'\\\nu'\end{pmatrix}]\,.\]
    The other cases are obvious for degree reasons.

 \vskip 2mm \noindent
    {\bf Step $2$. Verifying Equation \eqref{jacobi}:} Similar to approach of the first step, we proceed by verifying Equation \eqref{jacobi} for some particular choices of constant global sections. For all the other choices, the Jacobiator vanishes for degree reasons.\\
    
    \vskip 2mm \noindent Since restriction to $E_0$ coincides with the Lie algebroid of $\cG\Rightarrow \OO^2$, for three sections of degree $0$, the Jacobi identity is already satisfied.

\vskip 2mm \noindent For constant sections $\renewcommand{\arraystretch}{0.7}\begin{pmatrix}u\\v\end{pmatrix},\renewcommand{\arraystretch}{0.7}\begin{pmatrix}u'\\v'\end{pmatrix}\in \Gamma(E_0)$ and $\renewcommand{\arraystretch}{0.7}\begin{pmatrix}\mu\\a\\\nu\end{pmatrix}\in \Gamma(E_{-1})$, we claim that
\begin{align}\label{jac001}
    [\begin{pmatrix}u\\v\end{pmatrix},[\begin{pmatrix}u'\\v'\end{pmatrix},\begin{pmatrix}\mu\\a\\\nu\end{pmatrix}]]-[\begin{pmatrix}u'\\v'\end{pmatrix},[\begin{pmatrix}u\\v\end{pmatrix},\begin{pmatrix}\mu\\a\\\nu\end{pmatrix}]]-[[\begin{pmatrix}u\\v\end{pmatrix},\begin{pmatrix}u'\\v'\end{pmatrix}],\begin{pmatrix}\mu\\a\\\nu\end{pmatrix}]=\begin{pmatrix}0\\0\\0\end{pmatrix}\,.
\end{align}
We proceed by showing that each component of the result vanishes. Note that
\begin{align*}
    [\begin{pmatrix}u\\v\end{pmatrix},[\begin{pmatrix}u'\\v'\end{pmatrix},\begin{pmatrix}\mu\\a\\\nu\end{pmatrix}]]&=[\begin{pmatrix}u\\v\end{pmatrix},\begin{pmatrix}-2\la y,\overline{a}\!\cdot\!u'\ra+2\la y,v'\ra\mu\\x\!\cdot\!(\overline{u'}\!\cdot\!a)+(a\!\cdot\!v')\!\cdot\!\overline{y}-\mu (x\!\cdot\!\overline{v'})-\nu (u'\!\cdot\!\overline{y})\\-2\la x,a\!\cdot\!v'\ra+2\la x,u'\ra\nu\end{pmatrix}]\,,
\end{align*}

and 

\begin{align*}
    [[\begin{pmatrix}u\\v\end{pmatrix},\begin{pmatrix}u'\\v'\end{pmatrix}],\begin{pmatrix}\mu\\a\\\nu\end{pmatrix}]&=[(\la x,u\ra+\la y,v\ra)\begin{pmatrix}u'\\v'\end{pmatrix}-(\la x,u'\ra+\la y,v'\ra)\begin{pmatrix}u\\v\end{pmatrix},\begin{pmatrix}\mu\\a\\\nu\end{pmatrix}]\\
    &=(\la x,u\ra+\la y,v\ra)\begin{pmatrix}-2\la y,\overline{a}\!\cdot\!u'\ra+2\la y,v'\ra\mu\\x\!\cdot\!(\overline{u'}\!\cdot\!a)+(a\!\cdot\!v')\!\cdot\!\overline{y}-\mu (x\!\cdot\!\overline{v'})-\nu (u'\!\cdot\!\overline{y})\\-2\la x,a\!\cdot\!v'\ra+2\la x,u'\ra\nu\end{pmatrix}\\
    &-\{\begin{pmatrix}u\\v\end{pmatrix}\xleftrightarrow{} \begin{pmatrix}u'\\v'\end{pmatrix}\}\,.
\end{align*}
\vskip 2mm \noindent The first component of the left-hand side of Equation \eqref{jac001} then decomposes into
\begin{align*}
    =&-2\la \|y\|^2v+(y\!\cdot\!\overline{x})\!\cdot\!u-(\la x,u\ra+\la y,v\ra)y,\overline{a}\!\cdot\!u'\ra+2\la \|y\|^2v+(y\!\cdot\!\overline{x})\!\cdot\!u-(\la x,u\ra+\la y,v\ra)y,v'\ra\mu\\
    &-2\la y,((\overline{a}\!\cdot\!u')\!\cdot\!\overline{x}+y\!\cdot\!(\overline{v'}\!\cdot\!\overline{a})-\mu (v'\!\cdot\!\overline{x})-\nu (y\!\cdot\!\overline{u'}))\!\cdot\!u\ra+2\la y,v\ra(-2\la y,\overline{a}\!\cdot\!u'\ra+2\la y,v'\ra\mu)\\
    &-(\la x,u\ra+\la y,v\ra)(-2\la y,\overline{a}\!\cdot\!u'\ra+2\la y,v'\ra\mu)\\
    &-\{\begin{pmatrix}u\\v\end{pmatrix}\xleftrightarrow{} \begin{pmatrix}u'\\v'\end{pmatrix}\}\,.
\end{align*}

\vskip 2mm \noindent Straightforward cancellations then turn it into 
\begin{align*}
    =&-2\la \|y\|^2v+(y\!\cdot\!\overline{x})\!\cdot\!u,\overline{a}\!\cdot\!u'\ra+2\la (y\!\cdot\!\overline{x})\!\cdot\!u,v'\ra\mu\\
    &-2\la y,((\overline{a}\!\cdot\!u')\!\cdot\!\overline{x}+y\!\cdot\!(\overline{v'}\!\cdot\!\overline{a})-\mu (v'\!\cdot\!\overline{x})-\nu (y\!\cdot\!\overline{u'}))\!\cdot\!u\ra\\
    &+4\la x,u\ra(\la y,\overline{a}\!\cdot\!u'\ra-\la y,v'\ra\mu)\\
    &-\{\begin{pmatrix}u\\v\end{pmatrix}\xleftrightarrow{} \begin{pmatrix}u'\\v'\end{pmatrix}\}\,,
\end{align*}
\vskip 2mm \noindent which vanishes after the following consequences of Equations \eqref{semiasseqn2}, \eqref{switch1} and \eqref{switch2}:
\begin{align*}
    \la (y\!\cdot\!\overline{x})\!\cdot\!u,v'\ra&=\la -(y\!\cdot\!\overline{u})\!\cdot\!x+2\la x,u\ra y,v'\ra=-\la y, (v'\!\cdot\!x)\!\cdot\!u\ra+2\la x,u\ra\la y,v'\ra\,,\\
    \la y,((\overline{a}\!\cdot\!u')\!\cdot\!\overline{x})\!\cdot\!u\ra&=\la (y\!\cdot\!\overline{u})\!\cdot\!x,\overline{a}\!\cdot\!u'\ra=\la -(y\!\cdot\!\overline{x})\!\cdot\!u+2\la x,u\ra y,\overline{a}\!\cdot\!u'\ra\,,\\
    \la y, (y\!\cdot\!(\overline{v'}\!\cdot\!\overline{a}))\!\cdot\!u\ra&=\la \|y\|^2v',\overline{a}\!\cdot\!u\ra\,,\\
    \la y,(y\!\cdot\!\overline{u})\!\cdot\!u'\ra&=\|y\|^2\la u,u'\ra=\la y,(y\!\cdot\!\overline{u'})\!\cdot\!u\ra\,.
\end{align*}

\vskip 2mm \noindent The third component vanishes in a similar way. Finally, the second component equals to

\begin{align*}
&=(\|x\|^2u+(x\!\cdot\!\overline{y})\!\cdot\!v-(\la x,u\ra+\la y,v\ra)x)\!\cdot\!(\overline{u'}\!\cdot\!a)+(a\!\cdot\!v')\!\cdot\!(\|y\|^2\overline{v}+\overline{u}\!\cdot\!(x\!\cdot\!\overline{y})-(\la x,u\ra+\la y,v\ra)\overline{y})\\
&-\mu ((\|x\|^2u+(x\!\cdot\!\overline{y})\!\cdot\!v-(\la x,u\ra+\la y,v\ra)x))\!\cdot\!\overline{v'})-\nu (u'\!\cdot\!(\|y\|^2\overline{v}+\overline{u}\!\cdot\!(x\!\cdot\!\overline{y})-(\la x,u\ra+\la y,v\ra)\overline{y}))\\
&+x\!\cdot\!(\overline{u}\!\cdot\!(x\!\cdot\!(\overline{u'}\!\cdot\!a)+(a\!\cdot\!v')\!\cdot\!\overline{y}-\mu (x\!\cdot\!\overline{v'})-\nu (u'\!\cdot\!\overline{y})))+((x\!\cdot\!(\overline{u'}\!\cdot\!a)+(a\!\cdot\!v')\!\cdot\!\overline{y}-\mu (x\!\cdot\!\overline{v'})-\nu (u'\!\cdot\!\overline{y})))\!\cdot\!v)\!\cdot\!\overline{y}\\
&-(-2\la y,\overline{a}\!\cdot\!u'\ra+2\la y,v'\ra\mu)\!\cdot\!(x\!\cdot\!\overline{v})-(-2\la x,a\!\cdot\!v'\ra+2\la x,u'\ra\nu)\!\cdot\! (u\!\cdot\!\overline{y})\\
&-(\la x,u\ra+\la y,v\ra)\!\cdot\!(x\!\cdot\!(\overline{u'}\!\cdot\!a)+(a\!\cdot\!v')\!\cdot\!\overline{y}-\mu (x\!\cdot\!\overline{v'})-\nu (u'\!\cdot\!\overline{y}))\\
&-\{\begin{pmatrix}u\\v\end{pmatrix}\xleftrightarrow{} \begin{pmatrix}u'\\v'\end{pmatrix}\}\,,
\end{align*}
\vskip 2mm \noindent which can be rewritten as
\begin{align*}
&=\mu (-\|x\|^2u\!\cdot\!\overline{v'}-((x\!\cdot\!\overline{y})\!\cdot\!v)\!\cdot\!\overline{v'}-(x\!\cdot\!\overline{u}\!\cdot\!x)\!\cdot\!\overline{v'}-((x\!\cdot\!\overline{v'})\!\cdot\!v)\!\cdot\!\overline{y}-2\la y,v'\ra x\!\cdot\!\overline{v}+2\la x,u\ra x\!\cdot\!v'+2\la y,v\ra x\!\cdot\!\overline{v'})\\
&+\nu(-\|y\|u'\!\cdot\!\overline{v}-u'\!\cdot\!(\overline{u}\!\cdot\!(x\!\cdot\!\overline{y}))-x\!\cdot\!(\overline{u}\!\cdot\!(u'\!\cdot\!\overline{y}))-u'\!\cdot\!(\overline{y}\!\cdot\!v\!\cdot\!\overline{y})-2\la x,u'\ra u\!\cdot\!\overline{y}+2\la x,u\ra u'\!\cdot\!\overline{y}+2\la y,v\ra u'\!\cdot\!\overline{y})\\
&+(\|x\|^2u+x\!\cdot\!\overline{u}\!\cdot\!x)\!\cdot\!(u'\!\cdot\!a)+(a\!\cdot\!v')\!\cdot\!(\|y\|^2\overline{v}+\overline{y}\!\cdot\!v\!\cdot\!\overline{y})+((x\!\cdot\!\overline{y})\!\cdot\!v)\!\cdot\!(\overline{u'}\!\cdot\!a)+(a\!\cdot\!v')\!\cdot\!(\overline{u}\!\cdot\!(x\!\cdot\!\overline{y}))\\
&+x\!\cdot\!(\overline{u}\!\cdot\!((a\!\cdot\!v')\!\cdot\!\overline{y}))+((x\!\cdot\!(\overline{u'}\!\cdot\!a))\!\cdot\!v)\!\cdot\!\overline{y}+2\la y,\overline{a}\!\cdot\!u'\ra)x\!\cdot\!\overline{v}+2\la x,a\!\cdot\!v'\ra u\!\cdot\!\overline{y}\\
&-2(\la x,u\ra+\la y,v\ra)x\!\cdot\!(\overline{u'}\!\cdot\!a)-2(\la x,u\ra+\la y,v\ra)(a\!\cdot\!v')\!\cdot\!\overline{y}\\
&-\{\begin{pmatrix}u\\v\end{pmatrix}\xleftrightarrow{} \begin{pmatrix}u'\\v'\end{pmatrix}\}\,.
\end{align*}

\vskip 2mm \noindent Here, the coefficient of $\mu$ vanishes as a result of the following identities, concluded from Equation \eqref{semiasseqn2}:

\begin{align*}
    \|x\|^2u\!\cdot\!\overline{v'}+(x\!\cdot\!\overline{u}\!\cdot\!x)\!\cdot\!\overline{v'}&=2\la x,u\ra x\!\cdot\!\overline{v'}\,,\\
((x\!\cdot\!\overline{y})\!\cdot\!v)\!\cdot\!\overline{v'}&=2\la y,v\ra x\!\cdot\!\overline{v'}-2\la y,v'\ra x\!\cdot\!\overline{v}+((x\!\cdot\!\overline{v})\!\cdot\!v')\!\cdot\!\overline{y}\,.
\end{align*}

\vskip 2mm \noindent The coefficient of $\nu$ vanishes in a similar way. Using the identities $\|x\|^2u+x\!\cdot\!\overline{u}\!\cdot\!x=2\la x,u\ra x$ and $\|y\|^2\overline{v}+\overline{y}\!\cdot\!v\!\cdot\!\overline{y}=2\la y,v\ra \overline{y}$, the remaining terms simplify into

\begin{align*}
   &=((x\!\cdot\!\overline{y})\!\cdot\!v)\!\cdot\!(\overline{u'}\!\cdot\!a)+(a\!\cdot\!v')\!\cdot\!(\overline{u}\!\cdot\!(x\!\cdot\!\overline{y}))\\
&+x\!\cdot\!(\overline{u}\!\cdot\!((a\!\cdot\!v')\!\cdot\!\overline{y}))+((x\!\cdot\!(\overline{u'}\!\cdot\!a))\!\cdot\!v)\!\cdot\!\overline{y}+2\la y,\overline{a}\!\cdot\!u'\ra)x\!\cdot\!\overline{v}+2\la x,a\!\cdot\!v'\ra u\!\cdot\!\overline{y}\\
&-2\la y,v\ra x\!\cdot\!(\overline{u'}\!\cdot\!a)-2\la x,u\ra(a\!\cdot\!v')\!\cdot\!\overline{y}\\
&-\{\begin{pmatrix}u\\v\end{pmatrix}\xleftrightarrow{} \begin{pmatrix}u'\\v'\end{pmatrix}\}\,. 
\end{align*}

\vskip 2mm \noindent Again, iterative use of Equation \eqref{semiasseqn2} gives

\begin{align*}
    ((x\!\cdot\!\overline{y})\!\cdot\!v)\!\cdot\!(\overline{u'}\!\cdot\!a)&=2\la y,v\ra x\!\cdot\!(\overline{u'}\!\cdot\!a)-2\la y,\overline{a}\!\cdot\!u'\ra x\!\cdot\!\overline{v}+2\la u',a\!\cdot\!v\ra x\!\cdot\!\overline{y}-((x\!\cdot\!(\overline{u'}\!\cdot\!a))\!\cdot\!v)\!\cdot\!\overline{y}\,,\\
    (a\!\cdot\!v')\!\cdot\!(\overline{u}\!\cdot\!(x\!\cdot\!\overline{y}))&=2\la x,u\ra (a\!\cdot\!v')\!\cdot\!\overline{y}-2\la x,a\!\cdot\!v'\ra u\!\cdot\!\overline{y}+2\la u,a\!\cdot\!v'\ra x\!\cdot\!\overline{y}-x\!\cdot\!(\overline{u}\!\cdot\!((a\!\cdot\!v')\!\cdot\!\overline{y}))\,,
\end{align*}
\vskip 2mm \noindent which implies that the second component vanishes as well, and Equation \eqref{jac001} is verified.

\vskip 2mm \noindent For global constant sections $\renewcommand{\arraystretch}{0.7}\begin{pmatrix}u\\v\end{pmatrix}, \renewcommand{\arraystretch}{0.7}\begin{pmatrix}u'\\v'\end{pmatrix}\in \Gamma(E_0)$ and $t\in \Gamma(E_{-2})$, we have

\begin{align*}
    &[\begin{pmatrix}u\\v\end{pmatrix},[\begin{pmatrix}u'\\v'\end{pmatrix},t]]-[\begin{pmatrix}u'\\v'\end{pmatrix},[\begin{pmatrix}u\\v\end{pmatrix},t]]-[[\begin{pmatrix}u\\v\end{pmatrix},\begin{pmatrix}u'\\v'\end{pmatrix}],t]\\
    &=[\begin{pmatrix}u\\v\end{pmatrix},2(\la x,u'\ra+\la y,v'\ra)t]-[(\la x,u\ra+\la y,v\ra)\begin{pmatrix}u'\\v'\end{pmatrix},t]-\{\begin{pmatrix}u\\v\end{pmatrix}\xleftrightarrow{} \begin{pmatrix}u'\\v'\end{pmatrix}\}\\
    &=2(\la \|x\|^2u+(x\!\cdot\!\overline{y})\!\cdot\!v-(\la x,u\ra+\la y,v\ra)x,u'\ra+\la \|y\|^2v+(y\!\cdot\!\overline{x})\!\cdot\!u-(\la x,u\ra+\la y,v\ra)y,v'\ra)t\\
    &+4(\la x,u\ra+\la y,v\ra)(\la x,u'\ra+\la y,v'\ra)t-2(\la x,u\ra+\la y,v\ra)(\la x,u'\ra+\la y,v'\ra)t\\
    &-\{\begin{pmatrix}u\\v\end{pmatrix}\xleftrightarrow{} \begin{pmatrix}u'\\v'\end{pmatrix}\}=0
\end{align*}

\vskip 2mm \noindent which vanishes due to the following consequences of Equation \eqref{switch1}:

\begin{align*}
    &\la (x\!\cdot\!\overline{y})\!\cdot\!v,u'\ra=\la (y\!\cdot\!\overline{x})\!\cdot\!u',v\ra\,,\\
    &\la (x\!\cdot\!\overline{y})\!\cdot\!v',u\ra=\la(y\!\cdot\!\overline{x})\!\cdot\!u,v'\ra\,.
\end{align*}

\vskip 2mm \noindent Finally, for the constant sections $\renewcommand{\arraystretch}{0.7}\begin{pmatrix}u\\v\end{pmatrix}\in \Gamma(E_0)$ and $\renewcommand{\arraystretch}{0.7}\begin{pmatrix}\mu\\a\\\nu\end{pmatrix}, \renewcommand{\arraystretch}{0.7}\begin{pmatrix}\mu'\\a'\\\nu'\end{pmatrix}\in \Gamma(E_{-1})$ we have

\begin{align*}
    &[\begin{pmatrix}u\\v\end{pmatrix},[\begin{pmatrix}\mu\\a\\\nu\end{pmatrix},\begin{pmatrix}\mu'\\a'\\\nu'\end{pmatrix}]]-[\begin{pmatrix}\mu\\a\\\nu\end{pmatrix},[\begin{pmatrix}u\\v\end{pmatrix},\begin{pmatrix}\mu'\\a'\\\nu'\end{pmatrix}]]-[\begin{pmatrix}\mu'\\a'\\\nu'\end{pmatrix}[\begin{pmatrix}u\\v\end{pmatrix},\begin{pmatrix}\mu\\a\\\nu\end{pmatrix}]]\\
    &=[\begin{pmatrix}u\\v\end{pmatrix},\mu\nu'-\la a,a'\ra]-[\begin{pmatrix}\mu\\a\\\nu\end{pmatrix},\begin{pmatrix}-2\la y,\overline{a'}\!\cdot\!u\ra+2\la y,v\ra\mu'\\x\!\cdot\!(\overline{u}\!\cdot\!a')+(a'\!\cdot\!v)\!\cdot\!\overline{y}-\mu (x\!\cdot\!\overline{v})-\nu' (u\!\cdot\!\overline{y})\\-2\la x,a'\!\cdot\!v\ra+2\la x,u\ra\nu'\end{pmatrix}]\\
    &+\{\begin{pmatrix}\mu\\a\\\nu\end{pmatrix}\xleftrightarrow{} \begin{pmatrix}\mu'\\a'\\\nu'\end{pmatrix}\}\\
    &=2(\la x,u\ra+\la y,v\ra)(\mu\nu'-\la a,a'\ra)-2\mu(-\la x,a'\!\cdot\!v\ra+\la x,u\ra\nu')\\
    &-2\nu(-\la y,\overline{a'}\!\cdot\!u\ra+\la y,v\ra\mu')+2\la a,x\!\cdot\!(\overline{u}\!\cdot\!a')+(a'\!\cdot\!v)\!\cdot\!\overline{y}-\mu' (x\!\cdot\!\overline{v})-\nu' (u\!\cdot\!\overline{y})\ra\\
    &+\{\begin{pmatrix}\mu\\a\\\nu\end{pmatrix}\xleftrightarrow{} \begin{pmatrix}\mu'\\a'\\\nu'\end{pmatrix}\}\,,
\end{align*}

which simplifies into

\begin{align*}
    &=-2(\la x,u\ra+\la y,v\ra)\la a,a'\ra+2\la y,v\ra(\mu\nu'-\mu'\nu)+2\la a,x\!\cdot\!(\overline{u}\!\cdot\!a')+(a'\!\cdot\!v)\!\cdot\!\overline{y}\ra\\
    &+2(\mu\la x,a'\!\cdot\!v\ra+\nu\la y,\overline{a'}\!\cdot\!u\ra-\mu'\la x, a\!\cdot\!v\ra-\nu'\la y,\overline{a}\!\cdot\! u\ra)\\
    &+\{\begin{pmatrix}\mu\\a\\\nu\end{pmatrix}\xleftrightarrow{} \begin{pmatrix}\mu'\\a'\\\nu'\end{pmatrix}\}\,.
\end{align*}

By anti-symmetry, the second and the fourth terms of the expression cancel with the corresponding terms in $\{\renewcommand{\arraystretch}{0.7}\begin{pmatrix}\mu\\a\\\nu\end{pmatrix}\xleftrightarrow{} \renewcommand{\arraystretch}{0.7}\begin{pmatrix}\mu'\\a'\\\nu'\end{pmatrix}\}$. The remaining terms can be written as

\begin{align*}
    =2\la a,x\!\cdot\!(\overline{u}\!\cdot\!a')+u\!\cdot\!(\overline{x}\!\cdot\!a')-2\la x,u\ra a'\ra+2\la a,(a'\!\cdot\!y)\!\cdot\!\overline{v}+(a'\!\cdot\!v)\!\cdot\!\overline{y}-2\la y,v\ra a'\ra=0
\end{align*}
which vanishes as a result of Equation \eqref{semiasseqn2}.

\hfill $\blacksquare$
\end{proof}

\begin{rem}
In the study of octonionic particles \cite{C92}, a $Q$-manifold had been introduced which encodes the data of the Lie $3$-algebroid found independently here. 
\end{rem}

\begin{theorem}\label{mindim}
    The Lie algebroid $(E_0, [\cdot, \cdot], \rho)$ and correspondingly the Lie groupoid $\mathcal{G} \Rightarrow \mathcal{O}^2$ have the minimal dimension among Lie algebroids and Lie groupoids over $\mathcal{O}^2$, with $\mathcal{F}_{OH}$ as their orbits.
\end{theorem}

\begin{proof}
Since the universal Lie $3$-algebroid of $\mathcal{F}_{OH}$ introduced in Proposition \ref{thmlie3} is minimal at the origin, and $\mathrm{rank}(E_0) = 16$, Proposition \ref{minrank} implies that every such Lie algebroid is of dimension at least $16$. As a consequence, since $\mathcal{O}^2$ is a $16$-dimensional manifold, every Lie groupoid inducing $\mathcal{F}_{OH}$ has a manifold of arrows with dimension at least $32$.
    \hfill $\blacksquare$
\end{proof}

\appendix
\newpage\section{Computations with Macaulay2}\label{appa}

\vskip 2mm \noindent In this appendix, we use Eisenbud's Macaulay2 to construct a short exact sequence of \(\cO\)-modules, where \(\cO\) stands for the sheaf of smooth, polynomial, or real analytic functions on \(\OO^2 \cong \mathbb{R}^{16}\). This exact sequence will be used in Section \ref{secnonhom} to prove the maximality of \(\cF_{OH}\), and in Section \ref{lie3} to construct a universal Lie \(3\)-algebroid for \(\cF_{OH}\). This can be done as follows: In \href{https://www.unimelb-macaulay2.cloud.edu.au/}{Macaulay2Web}, first specify the ring of polynomial functions on \(\mathbb{R}^{16}\) by entering the following code in the first line:

\begin{verbatim}
    R = QQ[x_0, x_1, x_2, x_3, x_4, x_5, x_6, x_7, 
           y_0, y_1, y_2, y_3, y_4, y_5, y_6, y_7]
\end{verbatim}\,

Then, considering the characterizing equations of Lemma \ref{lemchar} for vector fields tangent to the leaves in \(\cL_{OH}\), we construct the morphism of \(\cO\)-modules $\mathrm{J}\colon \Gamma(\underline{\OO^2})\to \Gamma(\underline{\R\oplus\OO\oplus\R})$ given by
    \begin{equation*}
        \mathrm{J}\begin{pmatrix}u\\v\end{pmatrix}=\begin{pmatrix}\la x,u\ra\\u\!\cdot\!\overline{y}+x\!\cdot\!\overline{v}\\\la y,v\ra\end{pmatrix}
    \end{equation*}
    whose kernel gives the desired $\cO$-module of vector fields tangent to $\cL_{OH}$. In Macaulay2 one can use the following code to specify this morphism as a matrix, which is a result of considering the octonion multiplication as a product on $\R^8$:

\begin{verbatim}
    F = image matrix {{ x_0,  x_1,  x_2,  x_3,  x_4,  x_5,  x_6,  x_7, 
                        0  ,  0  ,  0  ,  0  ,  0  ,  0  ,  0  ,  0  },
                      { y_0,  y_1,  y_2,  y_3,  y_4,  y_5,  y_6,  y_7,  
                        x_0,  x_1,  x_2,  x_3,  x_4,  x_5,  x_6,  x_7},
                      {-y_1,  y_0, -y_3,  y_2, -y_5,  y_4,  y_7, -y_6, 
                        x_1, -x_0,  x_3, -x_2,  x_5, -x_4, -x_7,  x_6},
                      {-y_2,  y_3,  y_0, -y_1, -y_6, -y_7,  y_4,  y_5, 
                        x_2, -x_3, -x_0,  x_1,  x_6,  x_7, -x_4, -x_5},
                      {-y_3, -y_2,  y_1,  y_0, -y_7,  y_6, -y_5,  y_4, 
                        x_3,  x_2, -x_1, -x_0,  x_7, -x_6,  x_5, -x_4},
                      {-y_4,  y_5,  y_6,  y_7,  y_0, -y_1, -y_2, -y_3,
                        x_4, -x_5, -x_6, -x_7, -x_0,  x_1,  x_2,  x_3},
                      {-y_5, -y_4,  y_7, -y_6,  y_1,  y_0,  y_3, -y_2, 
                        x_5,  x_4, -x_7,  x_6, -x_1, -x_0, -x_3,  x_2},
                      {-y_6, -y_7, -y_4,  y_5,  y_2, -y_3,  y_0,  y_1, 
                        x_6,  x_7,  x_4, -x_5, -x_2,  x_3, -x_0, -x_1},
                      {-y_7,  y_6, -y_5, -y_4,  y_3,  y_2, -y_1,  y_0, 
                        x_7, -x_6,  x_5,  x_4, -x_3, -x_2,  x_1, -x_0},
                      { 0  ,  0  ,  0  ,  0  ,  0  ,  0  ,  0  ,  0  , 
                        y_0,  y_1,  y_2,  y_3,  y_4,  y_5,  y_6,  y_7}}

\end{verbatim}

\vskip 2mm \noindent This map can be completed to an exact sequence, using the code  

\begin{verbatim}
    resolution minimalPresentation F
\end{verbatim}

\vskip 2mm \noindent Finally, to view the matrices associated with the differential, use the codes

\begin{verbatim}
    o3.dd_1
\end{verbatim}
,
\begin{verbatim}
    o3.dd_2
\end{verbatim}
and 
\begin{verbatim}
    o3.dd_2
\end{verbatim}

This exact sequence can be translated back to the framework of octonions as follows

\begin{equation}
    \begin{tikzcd}
0\arrow[r] & \Gamma(E_{-2}) \arrow[r, "\exd^{(2)}"] & \Gamma(E_{-1}) \arrow[r, "\exd^{(1)}"] & \Gamma(E_{0}) \arrow[r, "\rho"] & \Gamma(\underline{\OO^2})\arrow[r, "\mathrm{J}"]&\Gamma(\underline{\R\oplus\OO\oplus\R})
\end{tikzcd}
\end{equation}

where

\begin{align*}
    E_{0}:=\underline{\OO^2},\qquad
    E_{-1}:=\underline{\R\oplus\OO\oplus\R},\qquad
    E_{-2}:=\underline{\R}\,.
\end{align*}

and for sections $\begin{pmatrix}
    u \\
    v
\end{pmatrix}\in \Gamma(E_{0})$, $\begin{pmatrix}
    \mu \\
    a   \\
    \nu
\end{pmatrix}\in \Gamma(E_{-1})$ and $t\in \Gamma(E_{-2})$:
\begin{align}\nonumber
\rho\begin{pmatrix}
u \\
v
\end{pmatrix}&:=\begin{pmatrix}
\|x\|^2u+(x\!\cdot\!\overline{y})\!\cdot\!v-(\la x,u\ra+\la y,v\ra)x \\
\|y\|^2v+(y\!\cdot\!\overline{x})\!\cdot\!u-(\la x,u\ra+\la y,v\ra)y
\end{pmatrix}\,,\\
    \exd^{(1)}\begin{pmatrix}
    \mu \\
    a   \\
    \nu 
\end{pmatrix}&:=\begin{pmatrix}
    \mu x+a\!\cdot\!y \\
    \nu y+\overline{a}\!\cdot\!x
\end{pmatrix}\,,\\
    \exd^{(2)}(t)&:=\begin{pmatrix}
    -\|y\|^2t \\
    (x\!\cdot\!\overline{y})t   \\
    -\|x\|^2t
\end{pmatrix}\,.\nonumber
\end{align}

Finally, using Malgrange's flatness theorem, this remains exact as a sequence of modules over real analytic or smooth functions.

\end{document}